\documentclass[a4paper,12pt, abstract=true]{scrartcl}
\usepackage[ansinew]{inputenc}
\usepackage[english]{babel}
\usepackage{latexsym,amssymb,amsfonts,bbm}
\usepackage{amsmath}
\usepackage{srcltx}
\usepackage{theorem}
\usepackage{xcolor}
\usepackage{graphicx}
\usepackage{esint}
\usepackage{stmaryrd}
\usepackage[normalem]{ulem}

\usepackage{hyperref}

\newtheorem{theorem}{Theorem}[section]
\newtheorem{lemma}{Lemma}[section]
\newtheorem{proposition}{Proposition}[section]
\newtheorem{corollary}{Corollary}[section]

\newtheorem{remark}{Remark}[section]

\newcommand{\proofendsign}{$\Box$} 
\newenvironment{proof}{{\noindent\textbf{Proof}}}{{\hspace*{\fill}\proofendsign\par\bigskip}}
\newenvironment{proofof}[1]{{\noindent\textbf{Proof\ of\ #1}}}{{\hspace*{\fill}\proofendsign\par\bigskip}}
\newenvironment{normalremark}{\begin{remark}\normalfont}{\end{remark}}

\newcommand{\cA}{{\cal A}}
\newcommand{\cB}{{\cal B}}
\newcommand{\cC}{{\cal C}}

\newcommand{\cE}{{\cal E}}
\newcommand{\cF}{{\cal F}}

\newcommand{\cU}{\cal U}

\newcommand{\cX}{{\cal X}}

\newcommand{\EEE}{\mathbb{E}}

\newcommand{\NNN}{\mathbb{N}}

\newcommand{\RRR}{\mathbb{R}}

\newcommand{\ZZZ}{\mathbb{Z}}
\newcommand{\pr}{\mathbb{P}}
\newcommand{\MP}{\mathbb{P}}

\newcommand{\eins}{\mathbbm{1}}

\newcommand{\OFP}{(\Omega,{\cal F},\MP)}

\def\bcswitch{\left\{\renewcommand{\arraystretch}{1.2}\begin{array}{c@{,~}c}}

\def\ecswitch{\end{array}\right.}

\def\bcswitch{\left\{\renewcommand{\arraystretch}{1.2}\begin{array}{c@{~,~}c}}

\def\ecswitch{\end{array}\right.}

\def\bcswitchs{\left\{\renewcommand{\arraystretch}{1.2}\begin{array}{c}}

\def\ecswitchs{\end{array}\right.}

\def\eps{\varepsilon}

\numberwithin{equation}{section}

\begin{document}

\title{A Kolmogorov-Chentsov type theorem on general metric spaces with applications to limit theorems for Banach-valued processes}

\author{Volker Kr\"atschmer\footnote{Faculty of Mathematics, University of Duisburg--Essen, {\texttt volker.kraetschmer@uni-due.de}}\and Mikhail Urusov\footnote{Faculty of Mathematics, University of Duisburg--Essen, {\texttt mikhail.urusov@uni-due.de}}}

\date{}

\maketitle

\begin{abstract}
The paper deals with moduli of continuity for paths of random processes
indexed by a general metric space $\Theta$
with values in a general metric space $\cX$. Adapting the moment condition on the increments from the classical Kolmogorov-Chentsov theorem, the obtained result on the modulus of continuity allows for H\"older-continuous modifications if the metric space $\cX$ is complete.
This result is universal in the sense that its applicability depends only on the geometry of the space $\Theta$. In particular, it is always applicable if $\Theta$ is a bounded subset of a Euclidean space
or a relatively compact subset of a connected Riemannian manifold.
The derivation is based on refined chaining techniques developed by Talagrand.
As a consequence of the main result a criterion is presented to guarantee uniform tightness of random processes with continuous paths.
This is applied to find central limit theorems for Banach-valued random processes.

\smallskip
\emph{Keywords}:
Kolmogorov-Chentsov type theorems, covering numbers, Talagrand's chaining technique, uniform tightness, Banach-valued central limit theorems.
\end{abstract}

\noindent
\textbf{Data Availability Statement:} No datasets were generated or analysed during the study

\section{Introduction and main result}
\label{path continuity}
Let $(\Theta,d_{\Theta})$ be a totally bounded metric space. For subsets $\overline{\Theta}\subseteq\Theta$ the diameter of $\overline{\Theta}$ w.r.t.\ $d_{\Theta}$ will be denoted by $\Delta(\overline{\Theta})$, whereas $N(\overline{\Theta},d_{\Theta},\eta)$ stands for the minimal number to cover $\overline{\Theta}$ with closed $d_{\Theta}$-metric balls of radius $\eta > 0$
with centers in $\overline\Theta$.
We will often need the following assumption on the geometry of~$\Theta$:
\begin{equation}
\label{size of parameter space}
\exists C,t > 0\;\forall\eta\in]0,\Delta(\Theta)]:N(\Theta,d_{\Theta},\eta)\leq C\eta^{-t}.
\end{equation}
Furthermore, let $(\cX,d_{\cX})$ be a metric space.
By $\cB(\cX)$ we denote the Borel $\sigma$-algebra on $\cX$. Let
$(X_{\theta})_{\theta\in\Theta}$ be an $\cX$-valued random process
on some probability space $\OFP$, i.e., for all $\theta\in\Theta$, $X_\theta$ is a random element in $(\cX,\cB(\cX))$.
Under a ``Kolmogorov-Chentsov type theorem'' we understand a theorem that, under an appropriate moment condition on the distance $d_{\cX}(X_{\theta},X_{\vartheta})$ for $\theta,\vartheta\in\Theta$, yields existence of a continuous or H\"older-continuous modification (cf.\ \cite{Chentsov1956}).
We establish the following general result.

\begin{theorem}
\label{Kolmogorov-Chentsov}
Assume~(\ref{size of parameter space}) and
\begin{equation}\label{eq:14122019a1}
(X_\theta,X_\vartheta)\text{ is }\cF|\cB(\cX^2)\text{-measurable for all pairs }(\theta,\vartheta)\in\Theta^2\text{ with }\theta\ne\vartheta.
\end{equation}
Let $M,p>0$ and $q>t$
(with $t$ from~(\ref{size of parameter space}))
be such that
\begin{equation}
\label{increment inequality}
\EEE\left[~d_{\cX}(X_{\theta},X_{\vartheta})^{p}~\right]\leq M d_{\Theta}(\theta,\vartheta)^{q}\quad\mbox{for}~\theta,\vartheta\in\Theta.
\end{equation}
Then, for any $\beta\in]0,(q - t)/p[$, there exists a finite constant
$\overline L(\Theta,C,t,p,q,\beta)$ dependent on $\Delta(\Theta), C, t, p, q$ and $\beta$ only
such that,
for every at most countable subset $\overline{\Theta}\subseteq\Theta$ with $\Delta(\overline\Theta) > 0$,
\begin{equation}
\label{Hoelder-constant2}
\EEE\left[\sup_{\theta,\vartheta\in\overline{\Theta}\atop \theta\not=\vartheta}~\frac{d_{\cX}(X_{\theta}, X_{\vartheta})^{p}}{d_{\Theta}(\theta,\vartheta)^{\beta p}}\right]\leq
L(\Theta,C,t,M,p,q,\beta):=M\overline L(\Theta,C,t,p,q,\beta).
\end{equation}
In particular, if $d_{\cX}$ is complete, then the random process
$(X_{\theta})_{\theta\in\Theta}$
has a modification which satisfies~(\ref{eq:14122019a1}) such that all its paths are H\"older-continuous of all orders $\beta\in]0,(q-t)/p[$.
\end{theorem}

\begin{normalremark}
Technical assumption~(\ref{eq:14122019a1})
is always satisfied when $\cX$ is a separable metric space
because, in this case, $\cB(\cX^2)=\cB(\cX)\otimes\cB(\cX)$.
In general, we only have the inclusion
$\cB(\cX^2)\supseteq\cB(\cX)\otimes\cB(\cX)$,
and the assumption is needed to ensure
measurability of $d_\cX(X_\theta,X_\vartheta)$.
\end{normalremark}

We consider Theorem~\ref{Kolmogorov-Chentsov} as our main ``building block''.
In the literature, Kolmogorov-Chentsov type theorems are sometimes formulated in a localized form.
A localized version of Theorem~\ref{Kolmogorov-Chentsov} where $\Theta$ is not necessarily totally bounded is presented in Section~\ref{sec:localized}.

\begin{normalremark}\label{rem:16122019a1}
The key assumption on the geometry of the parametric space $\Theta$ is~(\ref{size of parameter space}), where the value of $t$ is important, as we need to have $q>t$ in~(\ref{increment inequality}).\footnote{On the contrary, the value of $C$ in~(\ref{size of parameter space}) is not important.}
We remark that,
if $\Theta$ is a bounded subset of $\RRR^{m}$ with the Euclidean metric $d_{m,2} = d_{\Theta}$,
then (\ref{size of parameter space}) is always satisfied with $t=m$.\footnote{More precisely, we have
$N(\overline{\Theta},d_{\Theta},\eta)\leq \big[\big(8 \Delta(\Theta) + \eta\big)/\eta\big]^{m}$ for all $\eta >0$ (see \cite[Lemma 2.5]{vandeGeer2000}). In particular, (\ref{size of parameter space}) holds with $t = m$ and $C = 9 \Delta(\Theta)^{m}$.}
More generally, a relatively compact subset $\Theta$ of an $m$-dimensional connected Riemannian manifold always satisfies~(\ref{size of parameter space}) with $t=m$ (we provide more detail in Section~\ref{sec:Riemann}).
\end{normalremark}

In the classical formulation of the Kolmogorov-Chentsov theorem it is assumed that $\cX$ is a Banach space and $\Theta=[0,1]^m$ for some $m\in\NNN$ (see \cite[Theorem~I.2.1]{RevuzYor1999}), and the proof relies on the fact that the dyadic rationals are dense in $[0,1]$.
Since that time there appeared many other versions
of the Kolmogorov-Chentsov theorem
that essentially allow to treat more general sets $\Theta$.
We mention
\cite[Theorem~2.1]{MittmannSteinwart2003},
\cite[Theorem~3.9]{CoxEtAl2013},
\cite[Lemma~2.19]{GrohsEtAl2018},
\cite[Proposition~3.9]{HuddeEtAl2019}
for several recent formulations where $\Theta$ is a subset of $\RRR^m$.
Some versions of the Kolmogorov-Chentsov theorem only guarantee that
$\sup(d_\cX(X_\theta,X_\vartheta)/d_\Theta(\theta,\vartheta)^\beta)<\infty$ a.s.\
(i.e., it is not claimed that the expectation of the $p$-th power of that quantity is finite).
However, some applications such as the ones discussed in Sections \ref{asymptotic tightness random processes} and~\ref{CLT random processes} below require that the expectation is finite.
As another example of this kind we mention that the proof of Theorem~6.1 in \cite{AKU} would not work without finiteness of such an expectation
(see formula~(106) in \cite{AKU}).

In the aforementioned references,
$\cX$ is (a closed subset of) a Banach space
and all $X_\theta$ are assumed to be in $L^p$
(with $p$ from~(\ref{increment inequality})),
and the proof involves a certain extension result for Banach-valued H\"older-continuous mappings.
That extension result allows to pass from rectangular regions in $\RRR^m$ to general subsets $\Theta\subseteq\RRR^m$.
In our situation when $\cX$ is only a metric space and we do not assume
$\EEE[d_\cX(a,X_\theta)^p]<\infty$ for all $\theta$ and some $a\in\cX$
(or the like) such a method of the proof cannot work,
so we use essentially different ideas to prove Theorem~\ref{Kolmogorov-Chentsov}.

Another approach, used in \cite[Theorem~2.9]{Potthoff2009} (also see \cite[Corollary~4.3]{LangEtAl2016}),
is worth mentioning.
In that reference, the existence of a locally H\"older-continuous modification is proved for $\cX=\RRR$ under assumptions of a different kind.
In particular, the assumption on $\Theta$ is that it is a dyadically separable metric space.
The latter is a requirement of a different type than~(\ref{size of parameter space}) on the geometry of $\Theta$,
which allows to pursue the arguments initially elaborated for rectangular regions in $\RRR^m$ in more general situations.
The setup in \cite{Potthoff2009} is quite different from ours, and the relation between the approaches still has to be worked out.
Notice, however, that in the finite-dimensional situation $\Theta\subseteq\RRR^m$, the other approach imposes some restrictions on possible sets $\Theta$ (see \cite[Theorem~4.1]{Potthoff2009}), while our approach allows for arbitrary sets $\Theta\subseteq\RRR^m$ (see Proposition~\ref{lokalHoelder} and Remark~\ref{rem:270422a1} below).

We thus summarize the previous discussion
by noting that we obtain inequality~(\ref{Hoelder-constant2}),
essentially, only under requirement~(\ref{size of parameter space})
on the geometry of the metric space $\Theta$,
which is satisfied for bounded subsets of $\RRR^m$
(with $t=m$)
and allows to go beyond $\RRR^m$.
It is also worth noting that the right-hand side of~(\ref{Hoelder-constant2}) is the same for all countable subsets $\overline\Theta\subseteq\Theta$,
and that (\ref{Hoelder-constant2}) is the right way to formulate the result in the case when $d_\cX$ is incomplete
(and thus a continuous modification may fail to exist).

\medskip
In order to discuss applications of Theorem~\ref{Kolmogorov-Chentsov}, we formulate the following immediate

\begin{corollary}
\label{tightness with entropy}
Assume (\ref{size of parameter space}), (\ref{eq:14122019a1}) and~(\ref{increment inequality}).
Let $\beta\in ]0,(q-t)/p[$
(with $p$, $q$ from~(\ref{increment inequality}) and $t$ from~(\ref{size of parameter space})),
and let
$L(\Theta,C,t,M,p,q,\beta)$
be any constant satisfying~(\ref{Hoelder-constant2}).
Then, for every at most countable subset $\overline{\Theta}\subseteq\Theta$ and arbitrary $\delta>0$,
\begin{equation}
\label{bound lattice supremum}
\EEE\left[\sup_{\theta, \vartheta\in\overline{\Theta}\atop d(\theta,\vartheta)\leq\delta}d_{\cX}(X_{\theta},X_{\vartheta})^p\right] \leq 
L(\Theta,C,t,M,p,q,\beta)
\delta^{\beta p}.
\end{equation}
\end{corollary}

Notice that, like in Theorem~\ref{Kolmogorov-Chentsov}, inequality~(\ref{bound lattice supremum}) holds universally, i.e., independently of the random process satisfying (\ref{eq:14122019a1}) and~(\ref{increment inequality}).
This will turn out to be useful when analyzing weak convergence of $\cX$-valued random processes
(see Sections \ref{asymptotic tightness random processes} and~\ref{CLT random processes}).

\medskip
The crucial step for the proof of Theorem~\ref{Kolmogorov-Chentsov} is provided by the following auxiliary result. It is interesting in its own right.

\begin{lemma}
\label{Lemma B.2.7}
Assume (\ref{size of parameter space}), (\ref{eq:14122019a1}) and~(\ref{increment inequality}).
Let $\overline{\Theta}$ be some finite subset of $\Theta$ with $\Delta(\overline{\Theta}) > 0$.
Then, for any $\delta > 0$,
\begin{align*}
&\EEE\left[\sup_{\theta,\vartheta\in\overline{\Theta}\atop d_{\Theta}(\theta,\vartheta) \leq \delta}d_{\cX}(X_{\theta},X_{\vartheta})^{p}\right]\\
&\leq 
4^{ t + 2 p + 3 q + 2}\cdot M\cdot\left(N(\overline{\Theta},d_{\Theta},\delta/4)~
\left[\ln\big(N(\overline{\Theta},d_{\Theta},\delta/4)\big)\right]^{q}~\delta^{q}~ +~ 
\frac{C }{\big(2^{(q-t)/p} - 1\big)^{p}}~\delta^{q-t}\right). 
 \end{align*}
 \end{lemma}
 
In the case $\cX = \RRR$, Theorem~B.2.4 in~\cite{Talagrand2014} provides a result similar to Lemma~\ref{Lemma B.2.7}. For the proof, a refined chaining technique is used there, which we shall adopt to derive  Lemma~\ref{Lemma B.2.7}.

\medskip
The structure of the paper is as follows.
In Section~\ref{sec:localized} we discuss a localized version of Theorem~\ref{Kolmogorov-Chentsov} where $\Theta$ is not necessarily totally bounded.
As an example, in Section~\ref{sec:Riemann} we explicitly treat the case where $\Theta$ is a subset of a Riemannian manifold.
In Sections \ref{asymptotic tightness random processes} and~\ref{CLT random processes} we present some applications of Theorem~\ref{Kolmogorov-Chentsov} to weak convergence of Banach-valued processes.
Lemma~\ref{Lemma B.2.7} and Theorem~\ref{Kolmogorov-Chentsov} are proved in Section~\ref{proofs}.

\section{Localized version of Theorem~\protect\ref{Kolmogorov-Chentsov}}\label{sec:localized}
Since the literature in the case $\Theta\subseteq\RRR^m$ sometimes formulates Kolmogorov-Chentsov type theorems for unbounded $\Theta$ (by localizing the results of the type of Theorem~\ref{Kolmogorov-Chentsov}),
we now formulate and discuss the localized version of Theorem~\ref{Kolmogorov-Chentsov} for metric spaces $(\Theta,d_\Theta)$ that are not necessarily totally bounded.

The setting is as follows.
Let $(\Theta,d_\Theta)$ be a metric space satisfying

\smallskip\noindent
\textbf{Property~(P)}
There exists an increasing sequence $\{\Theta_n\}_{n\in\NNN}$, $\Theta_n\subseteq\Theta_{n+1}$, $n\in\NNN$, of totally bounded open subsets of $\Theta$ such that $\Theta=\bigcup_{n\in\NNN}\Theta_n$ and
\begin{equation}\label{eq:270422a1}
\forall n\in\NNN\;\exists C_n,t_n > 0\;\forall\eta\in]0,\Delta(\Theta_n)]:N(\Theta_n,d_{\Theta},\eta)\leq C_n\eta^{-t_n}.
\end{equation}

Let $(\cX,d_{\cX})$ be a complete metric space and let $(X_{\theta})_{\theta\in\Theta}$ be an $\cX$-valued random process
on some $\OFP$. 

\begin{proposition}\label{lokalHoelder}
Assume Property~(P), that the process $(X_\theta)_{\theta\in\Theta}$ satisfies~(\ref{eq:14122019a1}) and that, for all $n\in\NNN$, there exist $M_n,p_n, \rho_{n}>0$ and $q_n>t_n$ (with $t_n$ as in Property~(P)) such that
\begin{equation}
\label{vereinfachte Inkrementenbedingung}
\EEE\left[~d_{\cX}(X_{\theta},X_{\vartheta})^{p_n}~\right]\leq M_n~d_{\Theta}(\theta,\vartheta)^{q_n}\quad\text{for }\theta,\vartheta\in\Theta_n,~d_{\Theta}(\theta,\vartheta) < \rho_{n},\;n\in\NNN.
\end{equation}
Then the random process $(X_{\theta})_{\theta\in\Theta}$ has a modification $(\widetilde{X}_{\theta})_{\theta\in\Theta}$ satisfying~(\ref{eq:14122019a1}) such that all its paths are locally H\"older-continuous of all orders $\beta\in\bigcap_{n\in\NNN}[0,(q_n - t_n)/p_n[$,
where the expression ``H\"older-continuous of order $0$'' is understood as ``uniformly continuous''. Moreover, for $n\in\NNN$, $\overline{\theta}\in\Theta_{n}$, there is some open in $\Theta$ neighbourhood $V(\overline{\theta})$ of $\overline{\theta}$ such that
\begin{equation}\label{eq:011022a1}
\EEE\left[\sup_{\theta,\vartheta\in V(\overline{\theta}) \atop \theta\not=\vartheta}\frac{d_{\cX}(\widetilde{X}_{\theta},\widetilde{X}_{\vartheta})^{p_{n}}}{d_{\Theta}(\theta,\vartheta)^{\beta p_{n}}}\right] < \infty\quad\mbox{for all }\beta\in[0,(q_n - t_n)/p_n[.
\end{equation}
\end{proposition}

\begin{normalremark}\label{rem:270422a1}
Notice that any $\Theta\subseteq\RRR^m$ satisfies Property~(P).
We can take\footnote{We remark that, although such sets $\Theta_n$ are in general not open in $\RRR^m$, they are always open in $\Theta$, as required in Property~(P).}
$\Theta_n=\Theta\cap(\,]-n,n[^m)$ and all $t_n=m$, $n\in\NNN$, whereas the constants $C_n$ indeed depend on $n$.
Therefore, in the case of an arbitrary subset $\Theta\subseteq\RRR^m$ Proposition~\ref{lokalHoelder} includes, e.g., the following statement:
There is a locally H\"older-continuous modification
of all orders $\beta\in]0,(q-m)/p[$
whenever \eqref{vereinfachte Inkrementenbedingung} holds with
$\Theta_n=\Theta\cap(]-n,n[^m)$, $p_n=p>0$ and $q_n=q>m$ not depending on~$n$
(on the contrary, $M_n$ and $\rho_n$ are allowed to depend on~$n$).
Moreover, in this case, for any $\overline\theta\in\Theta$, there exists an open in $\Theta$ neighbourhood $V(\overline\theta)$ of $\overline\theta$
such that \eqref{eq:011022a1} with $p_n\equiv p$ holds for all $\beta\in]0,(q-m)/p[$.
\end{normalremark}

Although Proposition~\protect\ref{lokalHoelder} follows from Theorem~\ref{Kolmogorov-Chentsov} via standard arguments, we present a proof to make the paper self-contained.

\medskip\noindent
\textbf{Proof of Proposition~\ref{lokalHoelder}}
Fix any $n\in\NNN$.
The set $\Theta_{n}$ from Property~(P) is totally bounded.
Therefore, we can find open subsets $\Theta_{n,1},\ldots,\Theta_{n,r_{n}}$ of $\Theta$ with diameters less than $\rho_{n}$ such that 
$$
\Theta_{n} = \bigcup_{i=1}^{r_{n}}\overline\Theta_{n,i},
$$
where $\overline\Theta_{n,i}=\Theta_{n}\cap\Theta_{n,i}$.
By~(\ref{vereinfachte Inkrementenbedingung}) we can apply Theorem~\ref{Kolmogorov-Chentsov} on each $\overline\Theta_{n,i}$.
Hence each $(X_{\theta})_{\theta\in \overline\Theta_{n,i}}$ has a modification $(\overline{X}^{n,i}_{\theta})_{\theta\in\overline\Theta_{n,i}}$ which satisfies (\ref{eq:14122019a1}) such that all its paths are H\"older-continuous on $\overline\Theta_{n,i}$ of all orders $\beta\in[0,(q_n-t_n)/p_n[$ with 
$$
\EEE\left[\sup_{\theta,\vartheta\in\overline\Theta_{n,i}\atop \theta\ne\vartheta}\frac{d_{\cX}\big(\overline{X}^{n,i}_{\theta}, \overline{X}^{n,i}_{\vartheta}\big)^{p_{n}}}{d_{\Theta}(\theta,\vartheta)^{\beta p_{n}}}\right] < \infty\quad\text{for all }\beta\in[0,(q_n-t_n)/p_n[.
$$
If $i,j\in\{1,\ldots,r_n\}$ are such that $\overline{\Theta}_{n,i}\cap\overline{\Theta}_{n,j}\ne\emptyset$,
then the processes
$(\overline{X}^{n,i}_{\theta})_{\theta\in \overline{\Theta}_{n,i}\cap\overline{\Theta}_{n,j}}$ and
$(\overline{X}^{n,j}_{\theta})_{\theta\in \overline{\Theta}_{n,i}\cap\overline{\Theta}_{n,j}}$ are indistinguishable,
as they are both continuous, modifications of each other
and $\overline{\Theta}_{n,i}\cap\overline{\Theta}_{n,j}$ is separable (because totally bounded).
Using this, it is straightforward to construct  a modification $(\overline X^n_\theta)_{\theta\in\Theta_n}$ of $(X_\theta)_{\theta\in\Theta_n}$ which satisfies (\ref{eq:14122019a1}) such that all its paths are H\"older-continuous of all orders $\beta\in[0,(q_n-t_n)/p_n[$ on each $\overline\Theta_{n,i}$ with
\begin{equation}
\label{zweiStern}
\EEE\left[\sup_{\theta,\vartheta\in\overline\Theta_{n,i}\atop \theta\ne\vartheta}\frac{d_{\cX}\big(\overline{X}^{n}_{\theta}, \overline{X}^{n}_{\vartheta}\big)^{p_{n}}}{d_{\Theta}(\theta,\vartheta)^{\beta p_{n}}}\right] < \infty\quad\text{for all }\beta\in[0,(q_n-t_n)/p_n[\text{ and }i\in\{1,\ldots,r_{n}\}.
\end{equation}

Now we vary $n\in\NNN$.
Recall that $\Theta_n\subseteq\Theta_{n+1}$.
Since $(\overline X^n_\theta)_{\theta\in\Theta_n}$ and $(\overline X^{n+1}_\theta)_{\theta\in\Theta_n}$ are modifications of each other, both continuous and $\Theta_n$ is separable,
then $(\overline X^n_\theta)_{\theta\in\Theta_n}$ and $(\overline X^{n+1}_\theta)_{\theta\in\Theta_n}$ are indistinguishable.
Therefore, there exists $\Omega_n\in\cF$ with $\pr(\Omega_n)=1$ such that, for all $\omega\in\Omega_n$ and $\theta\in\Theta_n$, it holds
$\overline X^n_\theta(\omega)=\overline X^{n+1}_\theta(\omega)$.
We define $\Omega_\infty=\bigcap_{n\in\NNN}\Omega_n$ and 
the process $(\widetilde X_\theta)_{\theta\in\Theta}$ by the formula
\begin{equation*}
\label{Musterprozess}
\widetilde X_\theta(\omega)=\begin{cases}
\overline X^n_\theta(\omega),&\omega\in\Omega_\infty,\;\theta\in\Theta_n\setminus\Theta_{n-1},\;n\in\NNN,\\
\bar x,&\omega\notin\Omega_\infty,
\end{cases}
\end{equation*}
where $\Theta_0:=\emptyset$ and $\bar x\in\cX$ is arbitrary.
This is a modification of $(X_\theta)_{\theta\in\Theta}$ which satisfies (\ref{eq:14122019a1}), and all its paths are H\"older-continuous of all orders $\beta\in[0,(q_n - t_n)/p_n[$ on each $\overline\Theta_{n,i}$, $n\in\NNN$, $i\in\{1,\ldots,r_{n}\}$.
Recalling from Property~(P) that each $\Theta_n$ is open in $\Theta$, we get that each point $\theta\in\Theta$ belongs to some open subset $\overline\Theta_{n,i}$ of $\Theta$ (for some $n\in\NNN$ and $i\in\{1,\ldots,r_{n}\}$).
In particular, all paths of $(\widetilde X_\theta)_{\theta\in\Theta}$ are locally H\"older-continuous of all orders $\beta\in\bigcap_{n\in\NNN}[0,(q_n - t_n)/p_n[$, while the last statement of Proposition~\ref{lokalHoelder} follows from~\eqref{zweiStern}.\hfill$\Box$

\section{Example: Subsets of Riemannian manifolds}\label{sec:Riemann}

In this section we discuss applicability of Theorem~\ref{Kolmogorov-Chentsov} and Proposition~\ref{lokalHoelder} in the setting when $\Theta$ is a subset of an $m$-dimensional connected Riemannian manifold $M$.
More precisely, we are going to understand restrictions~(\ref{size of parameter space}) and Property~(P) on $\Theta$ in this setting.
Essentially, the results are:
\begin{itemize}
\item
Every relatively compact $\Theta\subseteq M$ satisfies~(\ref{size of parameter space}) with $t=m$ (Proposition~\ref{prop:250822a1});
\item
Every $\Theta\subseteq M$ satisfies Property~(P) with $t_n=m$, $n\in\NNN$ (Corollary~\ref{cor:250822a1}).
\end{itemize}
For basic concepts and results from differential geometry we refer to standard textbooks, e.g., \cite{Dieudonne1972}, \cite{doCarmo1992}, \cite{Jost2011} and \cite{KobayashiNomizu1963}.

\medskip
Let $(M,g)$ be any connected $m$-dimensional Riemannian manifold as defined in \cite{doCarmo1992}. This means that $M$ denotes an $m$-dimensional $C^{\infty}$-manifold endowed with the Riemannian metric $g$. By definition $g$ is a mapping which associates to each point $p\in M$ an inner product $g_{p}$ on the tangential space $T_{p}M$ at $p$ such that for $C^{\infty}$-vector fields $\mathcal{V}, \mathcal{W}$ on an open subset $G$ of $M$ the mapping
$$
G\to\RRR,~p\mapsto g_{p}(\mathcal{V}_{p},\mathcal{W}_{p}) 
$$
is differentiable of class $C^{\infty}$. Furthermore, let for $p,q$ denote by $\mathcal{C}_{pq}$ the set of all $C^{\infty}$-curves in $M$ joining $p$ to $q$. The length $L(c)$ of a curve $c\in\mathcal{C}_{pq}$ defined on the closed interval $I_{c}$ of $\RRR$ is
$$
L(c) := \int_{I_{c}}\sqrt{g_{c(t)}\big(c'(t),c'(t)\big)}~dt,
$$
where $c'(t)$ stands for the velocity of $c$ at $t$. Since $M$ is connected, the sets $\mathcal{C}_{pq}$ are always nonvoid (see \cite[p. 146]{doCarmo1992}), and the mapping
$$
d_{g}:M\times M\rightarrow\RRR,~(p,q)\mapsto\inf_{c\in \mathcal{C}_{pq}}L(c)
$$
is a metric on $M$ (see \cite[Proposition 7.2.5]{doCarmo1992}) sometimes called the \textit{inner metric} (\textit{induced by} $g$). Moreover, the topology induced by this metric coincides with the original topology on $M$ (see \cite[Proposition 7.2.6]{doCarmo1992}).

\begin{proposition}\label{prop:250822a1}
(i) Let $\Theta$ be any relatively compact subset of $M$.
Then there exist a compact subset 
$K_{m}$ of $\RRR^{m}$ as well as $r\in\NNN$ and $\delta > 0$ such that 
$$
N(\Theta,d_{g},\eta)\leq r N(K_{m},d_{m,2},\eta/\delta)\quad\text{for all }\eta > 0,
$$
where $d_{m,2}$ stands for the Euclidean metric on $\RRR^{m}$.
As a consequence, $\Theta$ satisfies condition~(\ref{size of parameter space}) with $t = m$ w.r.t.\ the metric $d_{g}$.

\smallskip
(ii) If $d_{g}$ is complete, then every $d_g$-bounded subset $\Theta$ of $M$ satisfies~(\ref{size of parameter space}) with $t = m$ w.r.t.\ the metric $d_{g}$.
\end{proposition}

\begin{corollary}\label{cor:250822a1}
Every $\Theta\subseteq M$ satisfies Property~(P) with $t_{n} = m$, $n\in\NNN$, w.r.t.\ the metric $d_{g}$.
\end{corollary}

\begin{proof}
Since $M$ is a $C^{\infty}$-manifold, we can find an open covering $\{\overline{\Theta}_n\}_{n\in\NNN}$ of $M$ consisting of relatively compact subsets of $M$ and satisfying $\overline{\Theta}_n\subseteq\overline{\Theta}_{n+1}$ for $n\in\NNN$ (see, e.g., \cite[(16.1.4)]{Dieudonne1972}).
By Proposition~\ref{prop:250822a1} this sequence of subsets satisfies~(\ref{eq:270422a1}) w.r.t.\ $d_{g}$ with $t_{n} = m$ for $n\in\NNN$ (and the constants $C_{n}$ indeed depend on $n$).
Hence every $\Theta\subseteq M$ satisfies Property~(P) with $t_{n} = m$, $n\in\NNN$, w.r.t.\ $d_{g}$, as we can choose\footnote{Notice that, although such sets $\Theta_n$ are in general not open in $M$, they are always open in $\Theta$, as required in Property~(P).}
$\Theta_{n} := \Theta\cap\overline{\Theta}_{n}$, $n\in\NNN$.
\end{proof}

In the rest of this section we prove Proposition~\ref{prop:250822a1}.
The proof is based on a couple of auxiliary results.

\begin{lemma}
\label{gutartige Kompaktheit}
Let $\overline{\Theta}$ be a nonvoid compact subset of $M$
and assume $\overline\Theta\subseteq G$, where $G$ is an open subset of $M$ allowing a chart $u\colon G\to\RRR^m$ which satisfies that $u(\overline{\Theta})$ is convex.
Then there is some $\delta > 0$ such that 
$$
N(\overline{\Theta},d_{g},\eta)\leq N\big(u(\overline{\Theta}),d_{m,2},\eta/\delta\big)\quad\mbox{for}~\eta > 0.
$$
\end{lemma}

\begin{proof}
Let $\{e_{1},\ldots,e_{m}\}$ stand for the standard basis on $\RRR^{m}$. For any $C^{\infty}$-mapping $g:\cU\rightarrow\RRR$ on some open subset $\cU$ of $\RRR^{m}$ we shall use notation $d_{x}g$ to denote the differential of $g$ at $x\in\cU$.

Let us introduce for $p\in G$ the set $\cC_{M}^{\infty}(p)$ of all real-valued $C^{\infty}$-mappings on some open neigbourhood of $p$. By definition, the tangential space $T_{p}M$ of $M$ at $p$ consists of real-valued mappings on $\cC_{M}^{\infty}(p)$. The chart $u$ provides the following basis of $T_{p}M$
$$
\frac{\partial}{\partial u_{i}}\big|_{p}: \cC_{M}^{\infty}(p)\rightarrow\RRR,~\varphi\mapsto d_{p}(\varphi\circ u^{-1})(e_{i})\quad(i\in\{1,\ldots,m\})
$$
(see \cite[p.8]{doCarmo1992}). Moreover, 
$$
\Big(\frac{\partial}{\partial u_{1}},\ldots,\frac{\partial}{\partial u_{m}}\Big): G\rightarrow\bigcup_{p\in G}T_{p}M,~p\mapsto \Big(\frac{\partial}{\partial u_{1}}\big|_{p},\ldots,\frac{\partial}{\partial u_{m}}\big|_{p}\Big)
$$
defines some $C^{\infty}$-vector field (see \cite[25f.]{doCarmo1992}).

Next, let for $x\in u(G)$ denote by $d_{x}u^{-1}$ the differential of $u^{-1}$ at $x$ which is a linear mapping from $\RRR^{m}$ into $T_{u^{-1}(x)}M$ satisfying
$$
d_{x}u^{-1}(e_{i}) = \frac{\partial}{\partial u_{i}}\big|_{u^{-1}(x)}\quad\mbox{for}~i=1,\ldots,m.
$$
Since $g_{u^{-1}(x)}$ is an inner product on $T_{u^{-1}(x)}M$, we may observe for any 
$v = (v_{1},\ldots,v_{m})\in\RRR^{m}$
$$
g_{u^{-1}(x)}\big(d_{x}u^{-1}(v),d_{x}u^{-1}(v)\big) = \sum_{i,j=1}^{m}v_{i}~ v_{j}~g_{u^{-1}(x)}\Big(\frac{\partial}{\partial u_{i}}\big|_{u^{-1}(x)},\frac{\partial}{\partial u_{j}}\big|_{u^{-1}(x)}\Big).
$$
Then, with $S^{m-1}$ denoting the Euclidean sphere in $\RRR^{m}$, we may conclude from the defining properties of the Riemannian metric $g$ that the mapping 
$$
f: u(\overline{\Theta})\times S^{m-1}\rightarrow\RRR,~(x,v)\mapsto\sqrt{g_{u^{-1}(x)}\big(d_{x}u^{-1}(v),d_{x}u^{-1}(v)\big)}
$$
is continuous with strictly positive outcomes. Moreover, its domain is a compact subset 
of $\RRR^{m}\times\RRR^{m}$ so that it attains its maximum $\delta$ which is a positive number.

Now, let $p, q\in \overline{\Theta}$ with $p\not= q$. Since $u(\overline{\Theta})$ is assumed to be convex, the mapping 
$$
\overline{c}:[0,1]\rightarrow \RRR^{m},~t\mapsto t u(q) + (1-t) u(p)
$$
is a $C^{\infty}$-curve in $\RRR^{m}$ satisfying $\overline{c}(t)\in u(\overline{\Theta})$ for $t\in [0,1]$. Then $c:= u^{-1}\circ\overline{c}\in\mathcal{C}_{pq}$, and by chain rule
$$
c'(t) = d_{\overline{c}(t)}u^{-1}\big(u(q) - u(p)\big)\quad t\in [0,1].
$$
Since $g_{u^{-1}\left(\overline{c}(t)\right)}$ is an inner product on $T_{u^{-1}\big(\overline{c}(t)\big)}M$ and $d_{\overline{c}(t)}u^{-1}$ is linear for every $t\in [0,1]$, we obtain 
\begin{align*}
\sqrt{g_{u^{-1}\left(\overline{c}(t)\right)}\big(c'(t),c'(t)\big)} 
&= 
\|u(p) - u(q)\|_{m,2} f\big(\overline{c}(t), [u(q) - u(p)]/\|u(p) - u(q)\|_{m,2}\big)\\
&
\leq \delta \|u(p) - u(q)\|_{m,2}\quad\mbox{for}~t\in [0,1],
\end{align*}
where $\|\cdot\|_{m,2}$ stands for the Euclidean norm on $\RRR^{m}$. Hence by definition of the inner metric $d_{g}$ we end up with 
$$
d_{g}(p,q)\leq L(c)\leq \delta \|u(p) - u(q)\|_{m,2}.
$$
Since $\delta$ does not depend on $p, q$, we now easily derive the claim of Lemma~\ref{gutartige Kompaktheit}.
\end{proof}

In the next step, using Lemma~\ref{gutartige Kompaktheit}, we prove the result of Proposition~\ref{prop:250822a1} first for compact subsets of~$M$.

\begin{lemma}
\label{Gutartigkeit mit Kompaktheit}
Let $\overline{\Theta}\subseteq M$ be nonvoid and compact. Then there exists a nonvoid compact subset $K_{m}$ of $\RRR^{m}$ as well as $r\in\NNN$ and $\delta > 0$ such that
$$
N(\overline{\Theta},d_{g},\eta)\leq r N\big(K_{m},d_{m,2},\eta/\delta\big)\quad\mbox{for}~\eta > 0.
$$
\end{lemma}

\begin{proof}
For any $p\in\overline{\Theta}$ we may find a chart $u_{p}$, defined on an open subset $G_{u_{p}}$ of $M$, and some $\varepsilon_{p} > 0$ such that $p\in G_{u_{p}}$ and 
$$
B_{\varepsilon_{p}}\big(u_{p}(p)\big) := \{x\in\RRR^{m}\mid d_{m,2}\big(x,u_{p}(p)\big)\leq \varepsilon_{p}\}\subseteq u_{p}(G_{u_{p}}).
$$
Setting $U_{\varepsilon_{p}}\big(u_{p}(p)\big) := \{x\in\RRR^{m}\mid d_{m,2}\big(x,u_{p}(p)\big) <  \varepsilon_{p}\}$ and
$G^{p} := u_{p}^{-1}\left(U_{\varepsilon_{p}}\big(u_{p}(p)\big)\right)$
we observe that $(G^{p})_{p\in\overline\Theta}$ is an open covering of $\overline{\Theta}$
because $U_{\varepsilon_{p}}\big(u_{p}(p)\big)$ is an open subset of $\RRR^{m}$. Hence by compactness of $\overline{\Theta}$ there exist $p_{1},\ldots,p_{r}\in M$ such that
$$
\overline{\Theta}\subseteq \bigcup_{i=1}^{r}G^{p_{i}}\subseteq \bigcup_{i=1}^{r}\Theta^{i},
$$
where $\Theta^{i} := u_{p_{i}}^{-1}\left(B_{\varepsilon_{p_{i}}}\big(u_{p_{i}}(p_{i})\big)\right)$
for $i=1,\ldots,r$. For any $i\in\{1,\ldots,r\}$ the set $\Theta^{i}$ meets the requirements of Lemma \ref{gutartige Kompaktheit}. Hence we may find $\delta_{1},\ldots,\delta_{r} > 0$ such that
$$
N(\Theta_{i},d_{g},\eta)\leq N\left(B_{p_{i}}\big(u_{p_{i}}(p_{i})\big),d_{m,2},\eta/\delta_{i}\right)\quad\mbox{for}~i\in\{1,\ldots,r\}, \eta > 0.
$$
The set 
$$
K_{m} := \bigcup_{i=1}^{r}B_{p_{i}}\big(u_{p_{i}}(p_{i})\big)
$$
is a compact subset of $\RRR^{m}$. Then setting,
$\delta := 4\max\{\delta_{1},\ldots,\delta_{r}\}$,
we end up with
\begin{align*}
N(\overline{\Theta},d_{g},\eta) 
\leq 
\sum_{i=1}^{r}N(\Theta^{i},d_{g},\eta/2)
\leq
\sum_{i=1}^{r}N\big(K_{m},d_{m,2},\eta/(4\delta_{i})\big)
\leq
r N\big(K_{m},d_{m,2},\eta/\delta\big)\quad\mbox{for}~\eta > 0.
\end{align*}
This completes the proof.
\end{proof}

Finally, we are ready to prove Proposition~\ref{prop:250822a1}.

\medskip\noindent
\textbf{Proof of Proposition~\ref{prop:250822a1}}\\[0.1cm]
(i) Let $\Theta$ be a nonvoid relatively compact subset of $M$. The topological closure $\overline{\Theta}$ is compact, and $N(\Theta,d_{g},\eta)\leq N(\overline{\Theta},d_{g},\eta/2)$ holds for every $\eta > 0$.
Therefore, the first claim immediately follows from Lemma~\ref{Gutartigkeit mit Kompaktheit}.

\smallskip\noindent
(ii) If $d_{g}$ is complete, then by the Hopf-Rinow theorem (see, e.g., \cite[Theorem 7.2.8]{doCarmo1992}) every $d_{g}$-bounded subset of $M$ is already relatively compact.
Therefore, the second claim follows from the first one.\hfill$\Box$

\section{Tightness for sequences of random processes}
\label{asymptotic tightness random processes}
Let $(\Theta,d_{\Theta})$ be a compact metric space and $(\cX,d_{\cX})$ a complete metric space. We denote by $\cC(\Theta,\cX)$ the space of all continuous mappings from $\Theta$ into $\cX$ endowed with uniform metric $d_{\infty}$ w.r.t.\ the metric $d_{\cX}$ and the induced Borel $\sigma$-algebra $\cB\big(\cC(\Theta,\cX)\big)$.

Some of the results we are going to present simplify in the case when $\cC(\Theta,\cX)$ is separable (hence Polish, as it is complete).
For some discussions below we recall that, as $\Theta$ is compact, $\cC(\Theta,\cX)$ is separable \emph{if and only if} $\cX$ is separable (see \cite[Lemma~3.99]{AliprantisBorder2006}).
We, however, stress at this point that we \emph{never assume} $\cX$ (equivalently, $\cC(\Theta,\cX)$) to be separable.

\medskip
Let us fix any sequence $(X_{n})_{n\in\NNN}$ of Borel random elements $X_{n}\colon\Omega\to\cC(\Theta,\cX)$ on some probability space $\OFP$.
We show how Corollary~\ref{tightness with entropy} leads to a sufficient condition for uniform tightness in $\cC(\Theta,\cX)$.

\begin{proposition}
\label{vereinfachtes Kriterium fuer Straffheit}
Let $\Theta$ fulfill property (\ref{size of parameter space}) with constants $C, t > 0$.
Let $\Theta'\subseteq\Theta$ be dense in $\Theta$.
Assume that
$$
\big(X_{n}(\cdot,\theta)\big)_{n\in\NNN}
\text{ is a uniformly tight sequence of random elements in }(\cX,\cB(\cX)),
\text{ for all }\theta\in\Theta',
$$
and that there exist $M,p>0$ and $q>t$ such that
\begin{equation}
\label{glm Inkremente Ungleichung}
\sup_{n\in\NNN}\EEE\left[~d_{\cX}\big(X_{n}(\cdot,\theta), X_{n}(\cdot,\vartheta)\big)^{p}~\right]\leq M~d_{\Theta}(\theta,\vartheta)^{q}\quad\mbox{for}~\theta, \vartheta\in\Theta.
\end{equation}
Then $(X_{n})_{n\in\NNN}$ is a uniformly tight sequence of Borel random elements in $\cC(\Theta,\cX)$.
\end{proposition}

We recall that (\ref{size of parameter space}) need not be assumed if $\Theta$ is a compact subset of $\RRR^m$ endowed with the Euclidean metric. 
In this case, it is enough only to require $q>m$ in~(\ref{glm Inkremente Ungleichung}) (see Remark~\ref{rem:16122019a1}).

\begin{normalremark}
Notice that \eqref{eq:14122019a1} is satisfied for all processes $X_n$ because they are assumed to be Borel random elements in $\cC(\Theta,\cX)$ in this section and the projection map
$$
\pi_{\theta,\vartheta}\colon\cC(\Theta,\cX)\to\cX^2,\quad f\mapsto(f(\theta),f(\vartheta)),
$$
is continuous for all $(\theta,\vartheta)\in\Theta^2$.
\end{normalremark}

\begin{normalremark}\label{rem:02112020a1}
Observe that, if $\cX$ is separable, then the statements

\smallskip\noindent
(A) $X_n\colon\Omega\to\cC(\Theta,\cX)$ is a \emph{Borel} random element, i.e., a random element in $\big(\cC(\Theta,\cX),\cB(\cC(\Theta,\cX))\big)$; and

\smallskip\noindent
(B) $X_n=(X_n(\cdot,\theta))_{\theta\in\Theta}$ is an $\cX$-valued process
(i.e., for all $\theta\in\Theta$, $X_n(\cdot,\theta)$ is a random element in $(\cX,\cB(\cX))$)
with continuous paths

\smallskip\noindent
are equivalent (see \cite[Lemma~14.1]{Kallenberg1997}).
Thus, whenever $\cX$ is a Polish space, in Proposition~\ref{vereinfachtes Kriterium fuer Straffheit} (and in what follows) we essentially work with sequences of continuous $\cX$-valued processes.
In general, when (A) and (B) no longer coincide, the right choice is always (A), i.e., always to consider Borel random elements in $\cC(\Theta,\cX)$, as the concept of tightness (in $\cC(\Theta,\cX)$) discussed in Proposition~\ref{vereinfachtes Kriterium fuer Straffheit} requires the Borel $\sigma$-algebra (in $\cC(\Theta,\cX)$).
\end{normalremark}

\begin{proofof}{Proposition~\ref{vereinfachtes Kriterium fuer Straffheit}}
We take an arbitrary $\beta\in]0,(q-t)/p[$.
By compactness of $\Theta$ there exists some at most countable dense subset $\overline\Theta$ of $\Theta$.
Corollary~\ref{tightness with entropy} together with the continuity of the processes $X_n$ yields, for all $\delta>0$ and $n\in\NNN$,
$$
\EEE\left[\sup_{\theta,\vartheta\in\Theta\atop d_{\Theta}(\theta,\vartheta)\leq\delta}d_{\cX}\big(X_{n}(\cdot,\theta), X_{n}(\cdot,\vartheta)\big)^p\right] 
= 
\EEE\left[\sup_{\theta,\vartheta\in\overline\Theta\atop d_{\Theta}(\theta,\vartheta)\leq\delta}d_{\cX}\big(X_{n}(\cdot,\theta),X_{n}(\cdot,\vartheta)\big)^p\right]
\le
L(\Theta,C,t,M,p,q,\beta)
\delta^{\beta p}.
$$
Using the Markov inequality, we conclude that, for every $\eps>0$,
$$
\lim_{\delta\to0+}\limsup_{n\to\infty}\,\pr\,\bigg(\sup_{\theta,\vartheta\in\Theta\atop d_{\Theta}(\theta,\vartheta)\leq\delta}d_{\cX}\big(X_{n}(\cdot,\theta), X_{n}(\cdot,\vartheta)\big)\ge\eps\bigg)=0.
$$
Now the criterion for uniform tightness in $\cC(\Theta,\cX)$ presented in Theorem~\ref{schwache Konvergenz im Pfadraum} applies and completes the proof.
\end{proofof}

We observe that essentially the same condition achieves rather different aims in Theorem~\ref{Kolmogorov-Chentsov} and in Proposition~\ref{vereinfachtes Kriterium fuer Straffheit}.
In Theorem~\ref{Kolmogorov-Chentsov}, condition~\eqref{increment inequality} ensures existence of a continuous modification for the process $X$ (when $\cX$ is complete, which is assumed in Section~\ref{asymptotic tightness random processes}),
while in Proposition~\ref{vereinfachtes Kriterium fuer Straffheit}, condition~\eqref{glm Inkremente Ungleichung} implies the uniform tightness in $\cC(\Theta,\cX)$ for the sequence $(X_n)$.
(Notice that \eqref{glm Inkremente Ungleichung} is nothing else but \eqref{increment inequality} required for all $X_n$ uniformly in $n$.)
It is, therefore, tempting to try to shift continuity of the processes into the conclusion of Proposition~\ref{vereinfachtes Kriterium fuer Straffheit}.
And, indeed, this easily follows from the discussions above,
although \emph{at the cost of requiring $\cX$ to be separable.}

\begin{corollary}\label{cor:25042021a1}
Assume that $\cX$ is separable.
Let $\Theta$ fulfill property~(\ref{size of parameter space}) with constants $C, t > 0$.
We consider a sequence $(X_n)_{n\in\NNN}$ of $\cX$-valued processes $X_n=(X_n(\cdot,\theta))_{\theta\in\Theta}$.
Let $\Theta'\subseteq\Theta$ be dense in $\Theta$.
Assume that $(X_n(\cdot,\theta))_{n\in\NNN}$ is a uniformly tight sequence of random elements in $(\cX,\cB(\cX))$,
for all $\theta\in\Theta'$, and that there exist $M,p>0$ and $q>t$ such that
\begin{equation}\label{eq:02112020a1}
\sup_{n\in\NNN}\EEE\left[~d_{\cX}\big(X_n(\cdot,\theta), X_n(\cdot,\vartheta)\big)^{p}~\right]\leq M~d_{\Theta}(\theta,\vartheta)^{q}\quad\mbox{for}~\theta, \vartheta\in\Theta.
\end{equation}
Then each process $X_n$ admits a modification $\overline X_n=(\overline X_n(\cdot,\theta))_{\theta\in\Theta}$ that has continuous paths $\theta\mapsto\overline X_n(\omega,\theta)$ for all $\omega\in\Omega$,
the processes $\overline X_n$, $n\in\NNN$, are Borel random elements in $\cC(\Theta,\cX)$,
and the sequence $(\overline X_n)_{n\in\NNN}$ is uniformly tight in $\cC(\Theta,\cX)$.
\end{corollary}

\noindent
\textbf{Proof}
Theorem~\ref{Kolmogorov-Chentsov} ensures the existence of the continuous modifications $\overline X_n$, $n\in\NNN$.
As $\cX$ is separable, then, due to the equivalence between (A) and~(B) in Remark~\ref{rem:02112020a1}, each $\overline X_n$ is a Borel random element in $\cC(\Theta,\cX)$.
The uniform tightness of the sequence $(\overline X_n)_{n\in\NNN}$ now follows from Proposition~\ref{vereinfachtes Kriterium fuer Straffheit}.\hfill$\Box$

\begin{normalremark}
If in Corollary~\ref{cor:25042021a1} we additionally require that each process $X_n$ is \emph{separable} (the definition is recalled below),
then we obtain that each process $X_n$ is itself continuous almost surely,
so that we obtain the uniform tightness for the sequence $(X_n)_{n\in\NNN}$ \emph{itself}.\footnote{Formally, we need to identify each process $X_n$ with almost all continuous paths with an indistinguishable process with all continuous paths, in order to view $X_n$ as a Borel random element in $\cC(\Theta,\cX)$.}
This immediately follows from Lemma~\ref{lem:02112020a1} below.
For when this remark can be useful, we observe that, in some situations, we are given processes that are \emph{a priori} separable (e.g., c\`adl\`ag $\cX$-valued processes in the case $\Theta=[0,1]$).
\end{normalremark}

It remains to justify the previous remark.
Recall that an $\cX$-valued process $(Y(\cdot,\theta))_{\theta\in\Theta}$ on some $(\Omega,\cF,\pr)$ is called \emph{separable}\footnote{In the sense of Doob.}
if there exist an at most countable subset $\Theta_0\subseteq\Theta$ dense in $\Theta$ and an event $\Omega_0\in\cF$ with $\pr(\Omega_0)=1$ such that for every open subset $\mathcal{G}$ of $\Theta$, and any closed subset $\mathcal{D}$ of $\mathcal{X}$ the following equality holds true
$$
\big\{\omega\in\Omega_{0}\mid Y(\omega,\theta)\in\mathcal{D}~\mbox{for all}~\theta\in\mathcal{G}\cap\Theta_{0}\big\} 
= 
\big\{\omega\in\Omega_{0}\mid Y(\omega,\theta)\in\mathcal{D}~\mbox{for all}~\theta\in\mathcal{G}\big\} 
$$
(see \cite{GikhmanSkorokhod1974}).

\begin{lemma}\label{lem:02112020a1}
Let $Y=\big(Y(\cdot,\theta)\big)_{\theta\in\Theta}$ be a separable $\cX$-valued process that admits a continuous modification. Then $Y=\big(Y(\cdot,\theta)\big)_{\theta\in\Theta}$ is itself continuous \emph{almost surely}, and hence there is an indistinguishable from $Y$ process $\widetilde Y$ such that \emph{all} its paths are continuous.
\end{lemma}

It is worth noting that, contrary to the general setting in Section~\ref{asymptotic tightness random processes}, for this lemma the metric space $\cX$ does not need to be complete.

\bigskip
\begin{proof}
Let $\widetilde Y=\big(\widetilde Y(\cdot,\theta)\big)_{\theta\in\Theta}$ be a continuous modification of $Y$, i.e., for all $\theta\in\Theta$ we have $\pr(\{\widetilde Y(\cdot,\theta)=Y(\cdot,\theta)\})=1$ and the paths $\theta\mapsto\widetilde Y(\omega,\theta)$ are continuous for all $\omega\in\Omega$. As $Y$ is separable, we can find an at most countable $\Theta_0\subseteq\Theta$ dense in $\Theta$ and $\Omega_0\in\cF$ with $\pr(\Omega_0)=1$ as described prior to Lemma~\ref{lem:02112020a1}.
Define
$$
\Omega_1=\bigcap_{\theta\in\Theta_0}\{\widetilde Y(\cdot,\theta)=Y(\cdot,\theta)\}\cap\Omega_0
$$
and observe that $\pr(\Omega_1)=1$. It suffices to show that $Y(\omega,\theta) = \widetilde{Y}(\omega,\theta)$ holds for $\omega\in\Omega_{1}$ and $\theta\in\Theta$. So let us fix 
$\omega\in\Omega_{1}$ and $\theta\in\Theta$.

For $k\in\NNN$ set $\mathcal{G}_{k} := \{\vartheta\in\Theta\mid d_{\Theta}(\theta,\vartheta) < 1/k\}$, and let $\mathcal{D}_{k}$ denote the closure of the set $\{Y(\omega,\vartheta)\mid\vartheta\in\mathcal{G}_{k}\cap\Theta_{0}\}$. Now, separability of $Y$ yields $Y(\omega,\theta)\in\mathcal{D}_{k}$. In particular, there is some sequence $(\vartheta^{k}_{n})_{n\in\NNN}$ in $\mathcal{G}_{k}\cap\Theta_{0}$ such that $Y(\omega,\vartheta^{k}_{n})\to Y(\omega,\theta)$, as $n\to\infty$. This implies $\widetilde{Y}(\omega,\vartheta^{k}_{n})\to Y(\omega,\theta)$, as $n\to\infty$, due to definition of $\Omega_{1}$. Moreover, we may select by compactness of $\Theta$ a subsequence $(\vartheta_{i(n)}^{k})_{n\in\NNN}$ of $(\vartheta^{k}_{n})_{n\in\NNN}$ which converges to some $\overline{\vartheta}^{k}\in\Theta$. Then, by continuity of  
$\widetilde Y$,
$$
Y(\omega,\theta)
=\lim_{n\to\infty}\widetilde Y(\omega,\theta^{k}_{i(k)})
=\widetilde Y(\omega,\overline{\vartheta}^{k}).
$$
As $d_{\Theta}(\theta,\overline{\vartheta}^{k})\leq 1/k$,
the sequence $(\overline{\vartheta}^{k})_{k\in\NNN}$ converges to $\theta$. Hence, drawing on the continuity of $\widetilde{Y}$ again, we end up with 
$$
\widetilde{Y}(\omega,\theta) = \lim_{k\to\infty}\widetilde Y(\omega,\overline{\vartheta}^{k}) = Y(\omega,\theta).
$$
This completes the proof.
\end{proof}

\section{Central limit theorems for Banach-valued random processes}
\label{CLT random processes}
Let $(\Theta,d_{\Theta})$ be a compact metric space, and let $(\cX,\|\cdot\|_{\cX})$ be a Banach space. We shall denote by $\cC(\Theta,\cX)$ the space of all continuous mappings from $\Theta$ into $\cX$. It will be endowed with sup-norm $\|\cdot\|_{\infty}$ w.r.t.\ $\|\cdot\|_{\cX}$, and the induced Borel $\sigma$-algebra $\cB\big(\cC(\Theta,\cX)\big)$. 
\medskip

Consider any i.i.d.\ sequence $(X_i)_{i\in\NNN}$  of Bochner-integrable Borel random elements in $\cC(\Theta,\cX)$ on some probability space $\OFP$. We want to investigate weak convergence of the sequence $(S_{n})_{n\in\NNN}$ consisting of Borel random elements in $\cC(\Theta,\cX)$ defined by
$$
S_{n} := \frac{1}{\sqrt{n}}\sum_{i=1}^{n}\big(X_{i} - \EEE^{B}[X_{i}]\big)\quad\mbox{for}~n\in\NNN,
$$
where $\EEE^{B}[X_{i}]$ denotes the Bochner-integral of $X_{i}$. 
We start with the following observation.

\begin{proposition}
\label{einfach tightness}
Let  $\|X_{1}\|_{\infty}$ be square integrable.

\medskip
(i) The following statements are equivalent:

\smallskip\noindent
a) The sequence $(S_{n})_{n\in\NNN}$ is uniformly tight;

\smallskip\noindent
b) The sequence $(S_{n})_{n\in\NNN}$ converges weakly to some centered Gaussian random element in $\big(\cC(\Theta,\cX),\cB(\cC(\Theta,\cX))\big)$.

\medskip
(ii) If the equivalent statements in part~(i) are satisfied, then the limiting law in~b) is tight.
\end{proposition}

We remark that, as every Borel probability measure in a Polish space is tight, statement~(ii) in Proposition~\ref{einfach tightness} has a message only when $\cC(\Theta,\cX)$ (equivalently, $\cX$) is non-separable.

\bigskip
\begin{proof}
As the Borel random element $X_1$ is Bochner-integrable, it is almost surely separably-valued. Then we can find a closed separable linear subspace $\widehat C$ of $\cC(\Theta,\cX)$ such that $\pr(\{X_1\in\widehat C\})=1$
(note that $\widehat C$ is itself a Polish space and $\widehat C\in\cB(\cC(\Theta,\cX))$).
It follows that $\EEE^{B}[X_{1}]\in\widehat C$. This yields $\pr(\{X_1-\EEE^{B}[X_{1}]\in\widehat C\})=1$, hence $\pr(\{S_n\in\widehat C\})=1$ for all $n\in\NNN$. In view of the portmanteau lemma this yields 
that every weak limit point  of the laws of $S_{n}$, $n\in\NNN$, is concentrated on $\widehat C$ (in particular, is tight), thus establishing part~(ii). 
Moreover, the implication $b)\Rightarrow a)$ in part~(i) now follows from Prokhorov's theorem, which applies due to the fact that all measures are concentrated on a Polish space.

\medskip
We turn to the implication $a)\Rightarrow b)$ in part~(i).
By Prokhorov's theorem, the uniformly tight sequence $(S_{n})_{n\in\NNN}$ is relatively weakly sequentially compact.
It remains to prove uniqueness of a limit point and its Gaussianity.
To this end, let $r\in\NNN$ and $\Lambda_j\colon\cC(\Theta,\cX)\to\RRR$, $j=1,\ldots,r$, be continuous linear functionals.
Classical multivariate central limit theorem applies to the sequence
$\big((\Lambda_{1}\circ S_{n},\ldots\Lambda_{r}\circ S_{n})\big)_{n\in\NNN}$
because
$$
\EEE\left[|\Lambda_j\circ X_1|^2\right]\le\|\Lambda_j\|^2\:\EEE\left[\|X_1\|_\infty^2\right]<\infty,\quad j=1,\ldots,r
$$
($\|\Lambda_j\|$ denotes the operator norm of $\Lambda_j$) and yields weak convergence to a centered Gaussian law in $\RRR^r$.
This identifies every weak limit point of the laws of $S_n$, $n\in\NNN$, as a Gaussian measure and uniquely determines every weak limit point on the $\sigma$-algebra $\cE$ generated by continuous linear functionals $\cC(\Theta,\cX)\to\RRR$.
Notice that $\cE\subseteq\cB(\cC(\Theta,\cX))$, and the inclusion can be strict (when $\cC(\Theta,\cX)$ is non-separable).
However, restricted to $\widehat C$ both $\sigma$-algebras coincide:
$$
\widehat C\cap\cE=\widehat C\cap\cB(\cC(\Theta,\cX))
$$
(see \cite[Theorem~I.2.1]{VakhaniaEtAl.1987}).
Recalling that every weak limit point is concentrated on $\widehat C$ completes the proof.
\end{proof}

For application of Proposition~\ref{einfach tightness} we can utilize our criterion in Proposition~\ref{vereinfachtes Kriterium fuer Straffheit}
and obtain the following result.

\begin{corollary}
\label{Hilfe von Kolmogorov-Chentsov}
Let $\Theta$ satisfy condition~(\ref{size of parameter space}) with constants $C,t > 0$, and let $\|X_{1}\|_{\infty}$ be square integrable.
Assume that there is a dense subset $\Theta'\subseteq\Theta$ such that 
\begin{equation}
\label{fdCLT-convergence}
\big(S_{n}(\cdot,\theta)\big)_{n\in\NNN}\text{ is a uniformly tight sequence of random elements in }\cX,\text{ for all }\theta\in\Theta',
\end{equation}
and that there exist $M, p > 0$ as well as $q > t$ with
\begin{equation}
\label{glm Inkremente Ungleichung CLT}
\sup_{n\in\NNN}\EEE\left[\|S_{n}(\cdot,\theta) - S_{n}(\cdot,\vartheta)\|_{\cX}^{p}\right]\leq M~d_{\Theta}(\theta,\vartheta)^{q}\quad\mbox{for}~\theta, \vartheta\in\Theta.
\end{equation}
Then the sequence $(S_{n})_{n\in\NNN}$ converges weakly to a tight centered Gaussian random element in $\big(\cC(\Theta,\cX),\cB\big(\cC(\Theta,\cX)\big)\big)$.
\end{corollary}

We want to discuss the requirements of Corollary~\ref{Hilfe von Kolmogorov-Chentsov} for special choices of the Banach space $\cX$.
Let us start with type~2 - Banach spaces. To recall, the Banach space $\cX$ is called a \emph{type~2 - Banach space} if there is a constant $C > 0$ such that,
for all $n\in\NNN$ and $\cX$-valued independent centered Borel random elements $W_{1},\dots,W_{n}$ such that $\|W_{i}\|_{\cX}$ are square integrable, we have the following inequality
$$
\EEE\left[\left\|\sum_{i=1}^{n}W_{i}\right\|_{\cX}^{2}\right] \leq C~\sum_{i=1}^{n}\EEE\left[\left\|W_{i}\right\|_{\cX}^{2}\right]
$$
(see, e.g., \cite[Theorem 2.1]{Hoffmann-JorgensenPisier1976}).
Prominent examples of type~2 - Banach space are the following:
\begin{itemize}
\item $\cX$ is a finite-dimensional vector space,
\item $\cX$ is an $L^{p}$-space on some $\sigma$-finite measure space $(\mathfrak{X},\cA,\nu)$ with $L^{p}$-norm $\|\cdot\|_{p}$ for $p\in [2,\infty[$ (see \cite[Section~9.2]{LedouxTalagrand1991}).
\end{itemize}
If $\cX$ is a type~2 - Banach space, then conditions (\ref{fdCLT-convergence}) and~(\ref{glm Inkremente Ungleichung CLT}) can be simplified in the following way.

\begin{proposition}
\label{CLTType2}
Let $\cX$ be a type~2 - Banach space, let $\Theta$ satisfy condition~(\ref{size of parameter space}) with constants $C,t > 0$, and let $\|X_{1}\|_{\infty}$ be square integrable. Then it holds: 
\begin{itemize}
\item [1)]
The sequence $(S_{n})_{n\in\NNN}$ always satisfies condition~(\ref{fdCLT-convergence}), even with $\Theta'=\Theta$.
\item [2)]
The sequence $(S_{n})_{n\in\NNN}$ satisfies condition~(\ref{glm Inkremente Ungleichung CLT}) with $p=2$ whenever there exist $M > 0$ and $q> t$ such that
\begin{equation}\label{eq:12042021a1}
\EEE\left[\|X_{1}(\cdot,\theta) - X_{1}(\cdot,\vartheta)\|_{\cX}^{2}\right]\leq M~d_{\Theta}(\theta,\vartheta)^{q}\quad\mbox{for}~\theta, \vartheta\in\Theta.
\end{equation}
\end{itemize}
In particular, under~\eqref{eq:12042021a1}, the sequence $(S_{n})_{n\in\NNN}$ converges weakly to a tight centered Gaussian random element in $\big(\cC(\Theta,\cX),\cB\big(\cC(\Theta,\cX)\big)\big)$.
\end{proposition}

It is worth noting that, even in the separable case, we cannot get this result from the general central limit theorem in type~2 - Banach spaces (see, e.g., \cite[Theorem~10.5]{LedouxTalagrand1991}) because, in Proposition~\ref{CLTType2}, it is only the space $\cX$ and not $\cC(\Theta,\cX)$ that has type~2.

\bigskip
\begin{proof}
Consider for $\theta\in\Theta$ the continuous linear operator $\pi_{\theta}\colon\mathcal{C}\big(\Theta,\cX\big)\rightarrow\cX$ defined by $\pi_{\theta}(f) := f(\theta)$. Then by Bochner-integrability of the Borel random element $X_1$ in $\cC(\Theta,\cX)$, we may conclude that the Borel random element $X_{1}(\cdot,\theta) = \pi_{\theta}\circ X_{1}$ of $\cX$ 
is Bochner-integrable with Bochner-integral $\EEE^{B}\big[X_{1}(\cdot,\theta)\big] = \pi_{\theta}\big(\EEE^{B}\big[X_{1}\big]\big)$. In particular, it 
is almost surely separably-valued. Hence the Borel random element $X_{1}(\cdot,\theta) - \EEE^{B}\big[X_{1}(\cdot,\theta)\big]$ is almost surely separably-valued too. This means that $X_{1}(\cdot,\theta) - \EEE^{B}\big[X_{1}(\cdot,\theta)\big]$ is concentrated on some separable closed subset of $\cX$. Due to completeness of $\|\cdot\|_{\infty}$ this implies that $X_{1}(\cdot,\theta) - \EEE^{B}\big[X_{1}(\cdot,\theta)\big]$ is a Radon Borel random element of $\cX$  (see \cite[p. 29, Corollary]{VakhaniaEtAl.1987}). Now,
statement~1) follows from the general central limit theorem in type~2 - Banach spaces (see \cite[Theorem 3.6]{Hoffmann-JorgensenPisier1976} or \cite[Theorem~10.5]{LedouxTalagrand1991}) along with the version of Prokhorov's theorem for Radon measures 
(see, e.g., \cite[Theorem~I.3.6]{VakhaniaEtAl.1987}).


Concerning statement~2), by the above definition of type~2 - Banach spaces, we can find some constant $C > 0$ such that 
$$
\sup_{n\in\NNN}\EEE\left[\|S_{n}(\cdot,\theta) - S_{n}(\cdot,\vartheta)\|_{\cX}^{2}\right]\leq {C}~ \EEE\left[\big\|X_{1}(\cdot,\theta)  - X_{1}(\cdot,\vartheta) - \EEE^{B}\big[X_{1}(\cdot,\theta) - X_{1}(\cdot,\vartheta)\big]\big\|_{\cX}^{2}\right].
$$
We now observe that
\begin{align*}
&
\EEE\left[\big\|X_{1}(\cdot,\theta) - X_{1}(\cdot,\vartheta) - \EEE^{B}\big[X_{1}(\cdot,\theta) - X_{1}(\cdot,\vartheta)\big]\big\|_{\cX}^{2}\right]\\ 
&\leq 
2\,\EEE\left[\big\|X_{1}(\cdot,\theta) -  X_{1}(\cdot,\vartheta)\big\|_{\cX}^{2}\right] 
+ 2 \left\|\EEE^{B}\big[X_{1}(\cdot,\theta)- X_{1}(\cdot,\vartheta)\big]\right\|_{\cX}^{2}\\
&\leq 
4\,\EEE\left[\big\|X_{1}(\cdot,\theta) -  X_{1}(\cdot,\vartheta)\big\|_{\cX}^{2}\right],
\end{align*}
where in the last step we use Jensen's inequality.
This completes the proof.
\end{proof}


Let us turn to cotype~2 - Banach spaces.
The Banach space $\cX$ is called a \emph{cotype~2 - Banach space} if there is a constant $C > 0$ such that, for all $n\in\NNN$ and $\cX$-valued independent centered Borel random elements
$W_{1},\dots,W_{n}$
such that $\|W_{i}\|_{\cX}$ are square integrable, we have the following inequality
$$
\EEE\left[\left\|\sum_{i=1}^{n}W_{i}\right\|_{\cX}^{2}\right] \geq C~\sum_{i=1}^{n}\EEE\left[\left\|W_{i}\right\|_{\cX}^{2}\right]
$$
(see, e.g., \cite{ChobanjanTarieladze1977}).
For a further preparation, let us also recall that a centered tight Borel random element $W$ in $\cX$ is called \emph{pre-gaussian} if there is some centered tight Gaussian random element $G$ in $\cX$ such that
$$
\EEE\left[L_{1}(W)~L_{2}(W)\right] = 
\EEE\left[L_{1}(G)~L_{2}(G)\right]
$$
holds for every pair $L_{1},L_{2}$ of continuous linear forms on $\cX$.

If $\cX$ is cotype~2 - Banach space we can obtain the following criterion for property~(\ref{fdCLT-convergence}).

\begin{proposition}
\label{CLTCoType2}
Let $\cX$ be a cotype~2 - Banach space, let $\Theta$ satisfy condition (\ref{size of parameter space}) with constants $C,t > 0$, and let $\|X_{1}\|_{\infty}$ be square integrable.
Assume that there is a dense subset $\Theta'\subseteq\Theta$ such that
$$
X_{1}(\cdot,\theta) - \EEE^{B}[X_{1}(\cdot,\theta)]
\quad\text{is pre-gaussian for all }\theta\in\Theta'.
$$
Then the sequence $(S_{n})_{n\in\NNN}$ satisfies property~(\ref{fdCLT-convergence}) (with this $\Theta'$).
\end{proposition}

\begin{proof}
First note that $X_1(\cdot,\theta)$ is a tight Borel random element in $\cX$ for every $\theta\in\Theta$ (cf.\ the proof of Proposition~\ref{CLTType2}).
Now the claim of Proposition~\ref{CLTCoType2} follows from the general central limit theorem in cotype~2 - Banach spaces (see \cite[Theorem~4.1]{ChobanjanTarieladze1977} or \cite[Theorem~10.7]{LedouxTalagrand1991}) along with the version of Prokhorov's theorem for Radon measures 
(see, e.g., \cite[Theorem~I.3.6]{VakhaniaEtAl.1987}).
\end{proof}

\begin{normalremark}
As a prominent example let $\cX$ be  an $L^{p}$-space on some $\sigma$-finite measure space $(\mathfrak{X},\cA,\nu)$ with $L^{p}$-norm $\|\cdot\|_{p}$ for $p\in [1,2]$. Then it is a cotype~2 - Banach space (see \cite[p.~188]{AraujoGine1980}). Moreover, for any $\theta\in\Theta$, the tight Borel random element $X_{1}(\cdot,\theta) - \EEE^{B}[X_{1}(\cdot,\theta)]$ is
pre-gaussian if and only if $L\circ \big(X_{1}(\cdot,\theta) - \EEE^{B}[X_{1}(\cdot,\theta)]\big)$ is square integrable for every continuous linear form $L$ on $\cX$, and
$$
\int_{\mathfrak{X}}\left(\EEE\left[\Big(X_{1}(\cdot,\theta)_{|x} - \EEE^{B}[X_{1}(\cdot,\theta)]_{|x}\Big)^{2}\right]\right)^{p/2}~\nu(dx) < \infty
$$
(see \cite[Theorem~11]{Jain1977}).
\end{normalremark}

\section{Proofs}\label{proofs}

Let us retake general assumptions and notations from Section~\ref{path continuity}. One key of our proofs is the following auxiliary technical result which extends Lemma~B.2.7 in \cite{Talagrand2014}. For a finite set $B$ we shall use notation ${\textrm card}(B)$ to denote  its cardinality.

\begin{lemma}
\label{TalagrandLemma B.2.7}
Let $\overline{\Theta}$ be some nonvoid finite subset of $\Theta$, and let $A\geq 1$ as well as $r\in\NNN$ such that $A^{r}\geq {\textrm card}(\overline{\Theta})$. Then for $c > 0$ there exists some $U\subseteq \overline{\Theta}\times\overline{\Theta}$ satisfying
\begin{eqnarray}
&&
\label{cardinality U} 
{\textrm card}(U)\leq A\cdot {\textrm card}(\overline{\Theta}).\\
&&
\label{diameter induced by U}
(\theta,\vartheta)\in U\quad\Rightarrow\quad d_{\Theta}(\theta,\vartheta)\leq c~r.\\
&&
\label{maximal increments}
\sup_{\theta,\vartheta\in\overline{\Theta}\atop d_{\Theta}(\theta,\vartheta)\leq c} d_{\cX}(X_{\theta},X_{\vartheta})\leq 2\sup_{(\theta,\vartheta)\in U} d_{\cX}(X_{\theta}, X_{\vartheta}).
\end{eqnarray}
\end{lemma}

\begin{proof}
According to the proof of Lemma~B.2.7 in \cite{Talagrand2014} we may find a sequence $(V_{l})_{l\in\NNN}$ of subsets of $\overline{\Theta}$, a sequence $(\theta_{l})_{l\in\NNN}$ in $\overline{\Theta}$ as well as a sequence $(r_{l})_{\in\NNN}$ in $\{1,\dots,r\}$ such that the following properties are satisfied
\begin{itemize}
\item $V_{1} = \overline{\Theta}$ and $\exists~ l_{0}\in\NNN~\forall~l\in\NNN,~l\geq l_{0}: V_{l} = \emptyset$.\\
\item $\theta_{l}\in V_{l}$ if $V_{l}\not=\emptyset$.\\
\item ${\textrm card}\big(\{\theta\in V_{l}\mid d_{\Theta}(\theta,\theta_{l})\leq r_{l} c\}\big)\leq A^{r_{l}}$ if $V_{l}\not=\emptyset$.\\
\item $V_{l+1} = V_{l}\setminus\{\theta\in V_{l}\mid d_{\Theta}(\theta,\theta_{l})\leq (r_{l} - 1) c\} = \{\theta\in V_{l}\mid d_{\Theta}(\theta,\theta_{l}) > (r_{l} - 1) c\}$ if $V_{l}\not=\emptyset$.\\
\item $\sum\limits_{l=1\atop V_{l}\not=\emptyset}^{\infty} A^{r_{l}}\leq A\cdot {\textrm card}(\overline{\Theta}).$
\end{itemize}
We shall show that the set
$$
U := \bigcup_{l= 1\atop V_{l}\not=\emptyset}^{\infty}\left\{(\theta_{l},\theta)\mid \theta\in V_{l},~d_{\Theta}(\theta_{l},\theta)\leq c r_{l}\right\}
$$
is as required.

\medskip 
First of all
\begin{eqnarray*}
{\textrm card}(U) 
&\leq& 
\sum_{l=1\atop V_{l}\not=\emptyset}^{\infty}{\textrm card}\left(\left\{(\theta_{l},\theta)\mid \theta\in V_{l},~d_{\Theta}(\theta_{l},\theta)\leq c r_{l}\right\}\right)\\
&=&
\sum_{l=1\atop V_{l}\not=\emptyset}^{\infty}{\textrm card}\left(\left\{\theta\in V_{l}\mid~d_{\Theta}(\theta_{l},\theta)\leq c r_{l}\right\}\right)
\leq \sum_{l=1\atop V_{l}\not=\emptyset}^{\infty} A^{r_{l}}\leq A\cdot {\textrm card}(\overline{\Theta})
\end{eqnarray*}
so that $U$ fulfills (\ref{cardinality U}). 

\medskip
Secondly, let $(\theta,\vartheta)\in U$. Then $\theta = \theta_{l}$ and $\vartheta\in V_{l}$ with $d_{\Theta}(\theta_{l},\vartheta)\leq c r_{l}$ for some $l\in\NNN$ with $V_{l}\not=\emptyset$. This means $d_{\Theta}(\theta,\vartheta)\leq c r$ because $r_{l}\leq r$. Thus (\ref{diameter induced by U}) holds for $U$. So it remains to show that (\ref{maximal increments}) is valid for $U$.

\medskip
Let $\theta,\vartheta\in\overline{\Theta}$ with $d_{\Theta}(\theta,\vartheta)\leq c$. By construction $\theta,\vartheta\in V_{1}$, whereas neither $\theta$ nor $\vartheta$ belongs to $V_{l}$ for $l\geq l_{0}$. So we may choose $l_{*} := \max\{l\in\NNN\mid \theta,\vartheta\in V_{l}\}$. Then $\theta\not\in V_{l_{*} + 1}$ or $\vartheta\not\in V_{l_{*} + 1}$, without loss of generality $\vartheta\not\in V_{l_{*} + 1}$. This means 
$d_{\Theta}(\theta_{l_{*}},\vartheta)\leq (r_{l_{*}} - 1) c$ so that also
$$
d_{\Theta}(\theta_{l_{*}},\theta)\leq d_{\Theta}(\theta_{l_{*}},\vartheta) + d_{\Theta}(\vartheta,\theta)\leq r_{l_{*}} c.
$$
Hence $(\theta_{l_{*}},\theta), (\theta_{l_{*}},\vartheta)\in U$, and thus
$$
d_{\cX}(X_{\theta},X_{\vartheta})\leq d_{\cX}(X_{\theta}, X_{\theta_{l_{*}}}) + 
d_{\cX}(X_{\theta_{l*}},X_{\vartheta})\leq 2\sup_{(\theta,\vartheta)\in U}d_{\cX}(X_{\theta},X_{\vartheta}).
$$
This shows (\ref{maximal increments}) and completes the proof.
\end{proof}

\subsection{Proof of Lemma~\protect\ref{Lemma B.2.7}}\label{proof of Lemma B.2.7}

In the first step we want to point out the central chaining argument that we shall use for the proof of Lemma~\ref{Lemma B.2.7}.

\begin{lemma}
\label{chaining lemma}
Let $\overline{\Theta}\subseteq\Theta$ be finite with at least two elements. Let $n_{0}$ be the largest element in $\ZZZ$ such that $\Delta(\overline{\Theta})\leq 2^{-n_{0}}$, and let 
$$
n_{1} := \min\left\{n\in\ZZZ\mid 2^{-n} < \inf_{\theta,\vartheta\in\overline{\Theta}\atop \theta\not=\vartheta}d_{\Theta}(\theta,\vartheta)\right\}.
$$
Then $n_{0} < n_{1}$, and the following statements are valid. 
\begin{itemize}
\item [1)] There exists a family $\{\Theta_{n}\mid n = n_{0},\dots,n_{1}\}$ of subsets of $\overline{\Theta}$ satisfying
\begin{eqnarray}
&&\label{Kardinalitaet der Chaining-Mengen}
{\textrm card}(\Theta_{n}) = N(\overline{\Theta},d_{\Theta},2^{-n})\quad\mbox{for}~n\in\{n_{0},\dots,n_{1}\},\\
&&
\label{Abstaende zu den Chaining-Mengen}
\inf_{\vartheta\in\Theta_{n}}d_{\Theta}(\theta,\vartheta) \leq 2^{-n}\quad\mbox{for}~n\in\{n_{0},\dots,n_{1}\}~\mbox{and}~\theta\in\overline{\Theta}.
\end{eqnarray}
\item [2)] The family $\{\Theta_{n}\mid n = n_{0},\dots,n_{1}\}$ from statement 1) may be associated with 
a family 
 $\{\varphi_{n}\mid n = n_{0},\dots,n_{1}\}$ of mappings $\varphi_{n}:\overline{\Theta}\rightarrow\Theta_{n}$ which fulfill the following properties:
 \begin{align}
 &\label{identity}
 \varphi_{n_{1}}:\overline{\Theta}\rightarrow\overline{\Theta},~\theta\mapsto\theta,\\
  & \label{constant mapping}
 \varphi_{n_{0}}\equiv\theta_{0}\quad\mbox{for some}~\vartheta_{0}\in\overline{\Theta}.\\[0.2cm]
 &
 \label{Abstand 1}
 d_{\Theta}\big(\varphi_{n+1}(\theta),\varphi_{n}(\theta)\big)\leq 2^{-n}\quad\mbox{for}~n\in\{n_{0},\dots,n_{1}-1\}~\mbox{and}~\theta\in\overline{\Theta},\\[0.2cm]
 &
 \label{Kardinalitaet Indexmengen}
 {\textrm card}\left(\left\{\big(\varphi_{n+1}(\theta),\varphi_{n}(\theta)\big)\mid \theta\in\overline{\Theta}\right\}\right)\leq  N(\overline{\Theta},d_{\Theta},2^{-(n + 1)})\quad\mbox{if}~n\in\{n_{0},\dots,n_{1}-1\}
 ,\\[0.2cm]
 &
 \label{Abstand 2}
 d_{\Theta}\big(\varphi_{n}(\theta),\varphi_{n}(\vartheta)\big)\leq 2^{-n + 2} + d_{\Theta}(\theta,\vartheta)\quad\mbox{for}~n\in\{n_{0},\dots,n_{1}\}~\mbox{and}~\theta, \vartheta\in\overline{\Theta}.
 \end{align}
 \item[3)] The chaining inequality
 \begin{equation}
 \label{chaining identity}
  d_{\cX}(X_{\theta},X_{\varphi_{n}(\theta)}) \leq \sum_{k = n}^{n_{1} - 1}d_{\cX}(X_{\varphi_{k+1}(\theta)},X_{\varphi_{k}(\theta)})\quad\mbox{for}~
 \theta\in\overline{\Theta}
  \end{equation} 
 is satisfied if $n\in\{n_{0},\dots,n_{1}-1\}$.
 \item [4)]
 Under assumptions (\ref{size of parameter space}) and (\ref{increment inequality}) from Theorem \ref{Kolmogorov-Chentsov} with $C > 0$, $q > t > 0$ the inequality
 \begin{equation}
 \label{Abschaetzung 1}
 \EEE\left[\sup_{\theta\in\overline{\Theta}} d_{\cX}(X_{\theta},X_{\varphi_{n}(\theta)})^{p}\right]
 \leq
 M~\left(\sum_{k=n}^{n_{1}-1} \frac{N(\overline{\Theta},d_{\Theta},2^{- (k+1)})^{1/p}}{2^{k q/p}}\right)^{p}
  \end{equation}
  holds for every $n\in\{n_{0},\dots,n_{1} - 1\}$. Furthermore 
  \begin{align}
  \label{Abschaetzung 1b}
\EEE\left[\sup_{\theta\in\overline{\Theta}}d_{\cX}(X_{\theta},X_{\varphi_{n}(\theta)})^{p}\right]
\leq 
\bcswitch
M~C~2^{2t}~ \dfrac{2^{(-n + 1) (q-t)}}{\big(2^{(q - t)/p} - 1\big)^{p}}&n_{1}\leq 0\\[0.3cm] 
M~C~2^{2t}~\Big(\dfrac{2^{(-n + 1) (q-t)/p} + 2^{(q-t)/p}}{2^{(q - t)/p} - 1}\Big)^{p}&n< 0  < n_{1}\\[0.3cm]
M~C~2^{q + t}~\dfrac{2^{-n (q-t)}}{\big(2^{(q - t)/p} - 1\big)^{p}}
&n \geq 0 
\ecswitch
\end{align}
for $n\in\{n_{0},\dots,n_{1} - 1\}$.
\end{itemize}
\end{lemma}

\begin{proof}
Statement 1) follows immediately from the definition of covering numbers. Furthermore, by construction we have
\begin{equation}
\label{Einrahmung Laufindex}
n_{0} < n_{1}\quad\mbox{and}\quad N(\overline{\Theta},d_{\Theta},2^{-n_{0}}) = 1,~N(\overline{\Theta},d_{\Theta},2^{-n_{1}}) = {\textrm card}(\overline{\Theta}).
\end{equation} 
Then the proof of statement 2) can be found in \cite[pp. 608f.]{Talagrand2014}. In view of (\ref{identity}), statement 3) may be verified easily by backward induction along with triangle inequality. So it remains to show statement 4).

\medskip
Let $n\in\{n_{0},\dots,n_{1}-1\}$. By chaining inequality (\ref{chaining identity}) we have 
\begin{eqnarray*}
\sup_{\theta\in\overline{\Theta}}d_{\cX}(X_{\theta},X_{\varphi_{n}(\theta)})^{p}
&\leq&
\sup_{\theta\in\overline{\Theta}}\Big(\sum_{k = n}^{n_{1} -1}d_{\cX}(X_{\varphi_{k+1}(\theta)},X_{\varphi_{k}(\theta)})\Big)^{p}\\
&\leq&
\Big(\sum_{k = n}^{n_{1} -1}\sup_{\theta\in\overline{\Theta}}d_{\cX}(X_{\varphi_{k+1}(\theta)}, X_{\varphi_{k}(\theta)})\Big)^{p}
\end{eqnarray*}
This implies by Minkowski's inequality
\begin{equation}
\label{allerersteAbschaetzung}
\left(\EEE\left[\sup_{\theta\in\overline{\Theta}}d_{\cX}(X_{\theta},X_{\varphi_{n}(\theta)})^{p}\right]\right)^{1/p} 
\leq 
\sum_{k = n}^{n_{1} -1}\left(\EEE\left[\sup_{\theta\in\overline{\Theta}}d_{\cX}(X_{\varphi_{k+1}(\theta)},X_{\varphi_{k}(\theta)})^{p}\right]\right)^{1/p}.
\end{equation}
Next, set for abbreviation 
$$
 I_{k} := \left\{\big(\varphi_{k+1}(\theta),\varphi_{k}(\theta)\big)\mid\theta\in\overline{\Theta}\right\}\quad\big(k\in\{n_{0},\dots,n_{1}-1\}\big).
 $$
Then we obtain in view of (\ref{increment inequality}) along with (\ref{Abstand 1}) and (\ref{Kardinalitaet Indexmengen})
 \begin{eqnarray*}
 \EEE\left[\sup_{(\alpha,\tilde{\alpha})\in I_{k}} d_{\cX}(X_{\alpha},X_{\tilde{\alpha}})^{p}\right]
 &\leq& 
 \sum_{(\alpha,\tilde{\alpha})\in I_{k}} \EEE\left[~d_{\cX}(X_{\alpha},X_{\tilde{\alpha}})^{p}~\right]\\
 &\stackrel{(\ref{Abstand 1}), (\ref{increment inequality})}{\leq}& 
 \frac{M~{\textrm card}(I_{k})}{2^{k q}}~ 
 \stackrel{(\ref{Kardinalitaet Indexmengen})}{\leq} 
\frac{M~ N(\overline{\Theta},d_{\Theta},2^{- (k+1)})}{2^{k q}}.
\end{eqnarray*}
By (\ref{allerersteAbschaetzung}) we end up with
\begin{eqnarray*}
\EEE\left[\sup_{\theta\in\overline{\Theta}}d_{\cX}(X_{\theta},X_{\varphi_{n}(\theta)})^{p}\right]
&\stackrel{(\ref{allerersteAbschaetzung})}{\leq}&\nonumber
\left(\sum_{k=n}^{n_{1}-1}\left(\EEE\left[\sup_{(\alpha,\tilde{\alpha})\in I_{k}} d_{\cX}(X_{\alpha},X_{\tilde{\alpha}})^{p}\right]\right)^{1/p}\right)^{p}\\
&\leq&
M~\left(\sum_{k=n}^{n_{1}-1} \frac{N(\overline{\Theta},d_{\Theta},2^{- (k+1)})^{1/p}}{2^{k q/p}}\right)^{p}. 
\end{eqnarray*}
 This shows (\ref{Abschaetzung 1}) of statement 4). For the remaining part of the proof we additionally assume that property (\ref{size of parameter space}) is satisfied with constants $C > 0, t \in ]0,q[$. Then we have 
$$
N(\overline{\Theta},d_{\Theta},2^{- (k+1)})\leq N(\Theta,d_{\Theta},2^{- (k+2)})\leq C~2^{t (k + 2)}\quad\mbox{for}~k\in\{n_{0},\dots,n_{1}-1\}.
$$ 
Note that $2^{- (k+1)} < \Delta(\overline{\Theta})\leq \Delta (\Theta)$ holds for every 
$k\in\{n_{0},\dots,n_{1}-1\}$ due to choice of $n_{0}$.
Now, (\ref{Abschaetzung 1b}) can be derived easily by routine calculations using geometric summation formulas.
This concludes the proof.
\end{proof}

\noindent
\textbf{Proof of Lemma~\ref{Lemma B.2.7}}
If $\delta <\inf\{d_{\Theta}(\theta,\vartheta)\mid \theta,\vartheta\in\overline{\Theta}, \theta\not=\vartheta\}$, then
$$
\big\{(\theta,\vartheta)\in\overline{\Theta}\times\overline{\Theta}\mid d_{\Theta}(\theta,\vartheta) \leq\delta\big\} 
= 
\big\{(\theta,\theta)\mid \theta\in\overline{\Theta}\big\}.
$$
In this case the statement of Lemma \ref{Lemma B.2.7} is trivial.

\medskip
From now on, let us assume $\delta \geq\inf\{d_{\Theta}(\theta,\vartheta)\mid \theta,\vartheta\in\overline{\Theta}, \theta\not=\vartheta\}$. In addition, let $n_{0}$ be the largest element in $\ZZZ$ such that $\Delta(\overline{\Theta})\leq 2^{-n_{0}}$, and let 
$$
n_{1} := \min\left\{n\in\ZZZ\mid 2^{-n} < \inf_{\theta,\vartheta\in\overline{\Theta}\atop \theta\not=\vartheta}d_{\Theta}(\theta,\vartheta)\right\}.
$$
We may find a family $\{\Theta_{n}\mid n = n_{0},\dots,n_{1}\}$ of subsets of $\overline{\Theta}$ and a family  
 $\{\varphi_{n}\mid n = n_{0},\dots,n_{1}\}$ of mappings $\varphi_{n}:\overline{\Theta}\rightarrow\Theta_{n}$ as in Lemma \ref{chaining lemma}.

\bigskip
If $N(\overline{\Theta},d_{\Theta},\delta/2) = 1$, then $\Delta(\overline{\Theta})\leq\delta$ so that 
$$
\EEE\left[\sup_{\theta,\vartheta\in\overline{\Theta}\atop d_{\Theta}(\theta,\vartheta)\leq\delta}d_{\cX}(X_{\theta}, X_{\vartheta})^{p}\right] 
= 
\EEE\left[\sup_{\theta,\vartheta\in\overline{\Theta}}d_{\cX}(X_{\theta},X_{\vartheta})^{p}\right],
$$ 
and in view of (\ref{constant mapping}) 
along with (\ref{Abschaetzung 1b})
\begin{eqnarray*}
\EEE\left[\sup_{\theta,\vartheta\in\overline{\Theta}}d_{\cX}(X_{\theta},X_{\vartheta})^{p}\right] 
&\leq& 
2^{p}~ \EEE\left[\sup_{\theta\in\overline{\Theta}}d_{\cX}(X_{\theta},X_{\varphi_{n_{0}}(\theta)})^{p}\right]\\
&\leq&
\bcswitch
M~C~2^{2t+p}~ \dfrac{2^{(-n_{0} + 1) (q-t)}}{\big(2^{(q - t)/p} - 1\big)^{p}}&n_{1}\leq 0\\[0.3cm] 
M~C~2^{2t+p}~\Big(\dfrac{2^{(-n_{0} + 1) (q-t)/p} + 2^{(q-t)/p}}{2^{(q - t)/p} - 1}\Big)^{p}&n_{0}< 0  < n_{1}\\[0.3cm]
M~C~2^{q+p + t}~\dfrac{2^{-n_{0} (q-t)}}{\big(2^{(q - t)/p} - 1\big)^{p}}
&n_{0} \geq 0 
\ecswitch.
\end{eqnarray*}
Moreover, by choice of $n_{0}$ we have 
$2^{-n_{0} + 1} < 4 \Delta (\overline{\Theta})\leq 4\delta $ so that 
routine calculations yield 
\begin{eqnarray*}
\EEE\left[\sup_{\theta,\vartheta\in\overline{\Theta}}d_{\cX}(X_{\theta},X_{\vartheta})^{p}\right] 
\leq 
\dfrac{2^{2 q + 2p}~M~C}{\big(2^{(q-t)/p} - 1\big)^{p}}~\delta^{q-t}.
\end{eqnarray*}
This shows Lemma \ref{Lemma B.2.7} in case of $N(\overline{\Theta},d_{\Theta},\delta/2) = 1$.

\bigskip 
Next, let us assume that  $N(\overline{\Theta},d_{\Theta},\delta/2)\geq 2$ is valid, and let us choose
$$
n_{2} := \max\{n\in\ZZZ\mid \delta\leq 2^{-n + 2}\}, ~n_{3} := n_{1}\wedge n_{2}\quad\mbox{and}\quad \overline{r} := \min\{r\in\NNN\mid 2^{r}\geq N(\overline{\Theta},d_{\Theta},\delta/4) \}.
$$
We have $2^{-n_{2} + 1} < \delta < 2 \Delta(\overline{\Theta})\leq 2^{-n_{0} + 1}$ so that $n_{2} > n_{0}$. By choice of $n_{2}$ we obtain
$$
\sup_{\theta,\vartheta\in\overline{\Theta}\atop d_{\Theta}(\theta,\vartheta)\leq\delta}d_{\cX}(X_{\theta}, X_{\vartheta}) 
\leq 
\sup_{\theta,\vartheta\in\overline{\Theta}\atop d_{\Theta}(\theta,\vartheta)\leq 2^{- n_{2} + 2}}d_{\cX}(X_{\theta}, X_{\vartheta}). 
$$
Moreover, for $\theta,\vartheta\in\overline{\Theta}$ with $d_{\Theta}(\theta,\vartheta)\leq 2^{- n_{2} + 2}$ we may further observe
\begin{eqnarray*}
d_{\cX}(X_{\theta},X_{\vartheta})
&\leq& 
d_{\cX}(X_{\theta},X_{\varphi_{n_{3}}(\theta)}) 
+ 
d_{\cX}(X_{\varphi_{n_{3}}(\theta)},X_{\varphi_{n_{3}}(\vartheta)}) 
+ 
d_{\cX}(X_{\vartheta},X_{\varphi_{n_{3}}(\vartheta)})\\
&\leq&
d_{\cX}(X_{\varphi_{n_{3}}(\theta)}, X_{\varphi_{n_{3}}(\vartheta)}) + 2~\sup_{\theta\in\overline{\Theta}}d_{\cX}(X_{\theta}, X_{\varphi_{n_{3}}(\theta)}).
\end{eqnarray*}
Then invoking (\ref{Abstand 2}), we obtain
$$
d_{\Theta}\big(\varphi_{n_{3}}(\theta), \varphi_{n_{3}}(\vartheta)\big) 
\leq 
2^{- n_{3} + 2} + d_{\Theta}(\theta,\vartheta) 
\leq 
2^{- n_{3} + 2} + 2^{- n_{2} + 2}\leq 2^{- n_{3} + 3}.
$$
Hence
\begin{align}
\label{Aufteilung der Ungleichung}
\sup_{\theta,\vartheta\in\overline{\Theta}\atop d_{\Theta}(\theta,\vartheta)\leq\delta}d_{\cX}(X_{\theta}, X_{\vartheta})
\leq 
\sup_{\theta,\vartheta\in\Theta_{n_{3}}\atop d_{\Theta}(\theta,\vartheta)\leq 2^{-n_{3} + 3}}d_{\cX}(X_{\theta}, X_{\vartheta}) 
+ 
2~\sup_{\theta\in\overline{\Theta}}d_{\cX}(X_{\theta}, X_{\varphi_{n_{3}}(\theta)}).
\end{align}
Furthermore, by (\ref{Kardinalitaet der Chaining-Mengen}) we may observe 
$$
2^{\overline{r}}\geq N(\overline{\Theta},d_{\Theta},\delta/4)\geq N(\overline{\Theta},d_{\Theta},2^{-n_{2}})\geq N(\overline{\Theta},d_{\Theta},2^{-n_{3}}) = {\textrm card}(\Theta_{n_{3}}).
$$
Therefore we may apply Lemma \ref{TalagrandLemma B.2.7} to $\Theta_{n_{3}}$ and $\overline{r}$, choosing 
$c := 2^{- n_{3} + 3}$ and $A = 2$. Hence we may find some $U\subseteq \Theta_{n_{3}}\times\Theta_{n_{3}}$ satisfying conditions (\ref{cardinality U}),  (\ref{diameter induced by U}) and (\ref{maximal increments}). Combination of    
(\ref{Aufteilung der Ungleichung}) with (\ref{maximal increments}) yields 
\begin{eqnarray*}
\sup_{\theta,\vartheta\in\overline{\Theta}\atop d_{\Theta}(\theta,\vartheta)\leq\delta}d_{\cX}(X_{\theta}, X_{\vartheta}) 
\leq 
2~ \sup_{(\theta,\vartheta)\in U}d_{\cX}(X_{\theta}, X_{\vartheta})~+~ 2~\sup_{\theta\in\overline{\Theta}}d_{\cX}(X_{\theta},X_{\varphi_{n_{3}}(\theta)})
\end{eqnarray*}
so that 
\begin{eqnarray*}
\sup_{\theta,\vartheta\in\overline{\Theta}\atop d_{\Theta}(\theta,\vartheta)\leq\delta}d_{\cX}(X_{\theta}, X_{\vartheta})^{p} 
\leq 
4^{p}~\left( \sup_{(\theta,\vartheta)\in U}d_{\cX}(X_{\theta}, X_{\vartheta})^{p}~+~ \sup_{\theta\in\overline{\Theta}}d_{\cX}(X_{\theta},X_{\varphi_{n_{3}}(\theta)})^{p}\right).
\end{eqnarray*}
Hence 
\begin{align}
&\nonumber
\EEE\left[\sup_{\theta,\vartheta\in\overline{\Theta}\atop d_{\Theta}(\theta,\vartheta)\leq\delta}d_{\cX}(X_{\theta}, X_{\vartheta})^{p}\right]\\ 
&\leq \label{Ausgangsungleichung}
4^{p}~\EEE\left[\sup_{(\theta,\vartheta)\in U} d_{\cX}(X_{\theta},X_{\vartheta})^{p}\right] + 4^{p}~\EEE\left[\sup_{\theta\in\overline{\Theta}} d_{\cX}(X_{\theta}, X_{\varphi_{n_{3}}(\theta)})^{p}\right].
\end{align}
If $n_{3} = n_{1}$, then $\varphi_{n_{3}}(\theta) = \theta$ for $\theta\in\overline{\Theta}$ due to (\ref{identity}). Hence
\begin{equation}
\label{unproblematische Ungleichung}
4^{p}~\EEE\left[\sup_{\theta\in\overline{\Theta}}d_{\cX}(X_{\theta}, X_{\varphi_{n_{3}}(\theta)})^{p}\right] = 0\quad\mbox{if}~n_{3} = n_{1}.
\end{equation}
So let us assume for a moment $n_{3} < n_{1}$. Then invoking property (\ref{size of parameter space}) and assumption (\ref{increment inequality}) we may conclude from Lemma \ref{chaining lemma}, (\ref{Abschaetzung 1b}) 
\begin{align}
\label{Abschaetzung 1c}
\EEE\left[\sup_{\theta\in\Theta}d_{\cX}(X_{\theta}, X_{\varphi_{n_{3}}(\theta)})^{p}\right] 
\leq 
\bcswitch
M~C~2^{2t}~ \dfrac{2^{(-n_{3} + 1) (q-t)}}{\big(2^{(q - t)/p} - 1\big)^{p}}&n_{1}\leq 0\\[0.3cm] 
M~C~2^{2t}~\Big(\dfrac{2^{(-n_{3} + 1) (q-t)/p} + 2^{(q-t)/p}}{2^{(q - t)/p} - 1}\Big)^{p}&n_{3}< 0  < n_{1}\\[0.3cm]
M~C~2^{q + t}~\dfrac{2^{-n_{3} (q-t)}}{\big(2^{(q - t)/p} - 1\big)^{p}}
&n_{3} \geq 0 
\ecswitch.
\end{align}
We also have $n_{3} = n_{2}$ so that the inequality $2^{-n_{3} + 1}\leq \delta$ is valid. 
Hence in view of (\ref{Abschaetzung 1c}) by easy calculations, we end up with
\begin{eqnarray}
4^{p}~\EEE\left[\sup_{\theta\in\overline{\Theta}}d_{\cX}(X_{\theta},X_{\varphi_{n_{3}}(\theta)})\right] 
&\leq& \label{problematischere Ungleichung}
\dfrac{2^{2t + 3 p}~M~C}{\big(2^{(q-t)/p} - 1\big)^{p}}~\delta^{q - t}\quad\mbox{if}~n_{3} < n_{1}.
\end{eqnarray}
Furthermore, applying sequentially (\ref{increment inequality}), (\ref{diameter induced by U}), (\ref{cardinality U}) and (\ref{Kardinalitaet der Chaining-Mengen}) we may observe \begin{eqnarray*}
\EEE\left[\sup_{(\theta,\vartheta)\in U}
d_{\cX}(X_{\theta},X_{\vartheta})^{p}
\right]
&\leq& \nonumber
\sum_{(\theta,\vartheta)\in U}
\EEE\left[~d_{\cX}(X_{\theta},X_{\vartheta})^{p}~\right]\\
&\stackrel{(\ref{increment inequality})}{\leq}& \nonumber
M~\sum_{(\theta,\vartheta)\in U}d_{\Theta}\big(\theta,\vartheta\big)^{q}\\
&\stackrel{(\ref{diameter induced by U})}{\leq}& \nonumber
M~{\textrm card}(U)~2^{q~(- n_{3} + 3)}~\overline{r}^{q}\\ 
 &\stackrel{(\ref{cardinality U})}{\leq}& \nonumber
2 M~ {\textrm card}\big(\Theta_{n_{3}}\big)~2^{q~(- n_{3} + 3)}~\overline{r}^{q}\\
&\stackrel{(\ref{Kardinalitaet der Chaining-Mengen})}{\leq}& 
2 M~N(\overline{\Theta},d_{\Theta},2^{-n_{3}})~2^{q~(- n_{3} + 3)}~\overline{r}^{q}\\
&\leq&
2 M~N(\overline{\Theta},d_{\Theta},2^{-n_{2}})~2^{q~(- n_{3} + 3)}~\overline{r}^{q}
\end{eqnarray*}
By choice of $n_{1}$ and $\delta$ we have $2^{-n_{3} + 3} = 2^{-n_{1} + 3} < 16 \delta$ if $n_{3} < n_{2}$. Otherwise, we obtain 
$2^{-n_{3} + 3} = 2^{-n_{2} + 1} 4 < 4 \delta$ due to definition of $n_{2}$. In addition $2^{-n_{2}}\geq\delta/4$. Hence 
$$
\EEE\left[\sup_{(\theta,\vartheta)\in U}
d_{\cX}(X_{\theta},X_{\vartheta})^{p}\right] 
\leq 
2 M~N(\overline{\Theta},d_{\Theta},\delta/4)~\delta^{q}~ (16\overline{r})^{q}
$$
The choice of $\overline{r}$ implies $2^{\overline{r}- 1} < N(\overline{\Theta},d_{\Theta},\delta/4)$ so that 
$$
\overline{r}\leq\dfrac{2\ln\big(N(\overline{\Theta},d_{\Theta},\delta/4)\big)}{\ln(2)}.
$$
Therefore
\begin{equation}
\label{naechste Abschaetzung}
4^{p}~ \EEE\left[\sup_{(\theta,\vartheta)\in U}
d_{\cX}(X_{\theta},X_{\vartheta})^{p}\right] 
\leq 
4^{p + 3 q + 1}~ 
M~N(\overline{\Theta},d_{\Theta},\delta/4)~
\left[\ln\big(N(\overline{\Theta},d_{\Theta},\delta/4)\big)\right]^{q}~\delta^{q}.
\end{equation}
Putting (\ref{Ausgangsungleichung}), (\ref{unproblematische Ungleichung}), (\ref{problematischere Ungleichung}) and~(\ref{naechste Abschaetzung}) together, we now easily derive the statement of Lemma~\ref{Lemma B.2.7} if $N(\overline{\Theta},d_{\Theta},\delta/2)\geq 2$. The proof is complete.\hfill$\Box$

\subsection{Proof of Theorem~\protect\ref{Kolmogorov-Chentsov}}\label{proof of main result}

Let (\ref{size of parameter space}) be satisfied with constants $C, t > 0$, and let $(X_{\theta})_{\theta\in\Theta}$ fulfill inequality (\ref{increment inequality}) with constants $M > 0, q >t$. Moreover, let us fix $\beta\in ]0,(q-t)/p[$. First, we want to show inequality (\ref{Hoelder-constant2}) for finite subsets of $\Theta$.

\begin{proposition}
\label{Kolmogorov-Chentsov_finite}
There exists a finite constant
$L(\Theta,C,t,M,p,q,\beta)$ that depends on
$\Delta(\Theta)$, $C$, $t$, $M$, $p$, $q$ and $\beta$ only
such that, for any finite subset $\overline{\Theta}\subseteq\Theta$ with at least two elements, it holds
$$
\EEE\left[\sup_{\theta,\vartheta\in\overline{\Theta}\atop \theta\not=\vartheta}~\frac{d_{\cX}(X_{\theta}, X_{\vartheta})^{p}}{d_{\Theta}(\theta,\vartheta)^{\beta p}}\right]\leq
L(\Theta,C,t,M,p,q,\beta).
$$
Moreover, $L(\Theta,C,t,M,p,q,\beta)$ can be chosen to depend linearly on $M$:
$L(\Theta,C,t,M,p,q,\beta)=M\overline L(\Theta,C,t,p,q,\beta)$.
\end{proposition}

\begin{proof}
Let $\overline{\Theta}$ be any finite subset of $\Theta$ with at least two elements. Set $\eta_{k} := 2^{-k}\big(\Delta(\Theta) + 1\big)$ for $k\in\NNN$, and let the set $J$ be defined to consist of all $k\in\NNN$ with 
$\eta_{k} < d_{\Theta}(\theta,\vartheta)\leq 2\eta_{k}$ for some $\theta,\vartheta\in\overline{\Theta}$. Note $J\not=\emptyset$. Then
\begin{align}
\EEE\left[\sup_{\theta,\vartheta\in\overline{\Theta}\atop \theta\not=\vartheta}~\frac{d_{\cX}(X_{\theta}, X_{\vartheta})^{p}}{d_{\Theta}(\theta,\vartheta)^{\beta p}}\right]
&\nonumber \leq 
\sum_{k\in J}\EEE\left[\sup~\left\{\frac{d_{\cX}(X_{\theta},X_{\vartheta})^{p}}{d_{\Theta}(\theta,\vartheta)^{\beta p}}~\Big|~ \theta,\vartheta\in\overline{\Theta},~\eta_{k}< d_{\Theta}(\theta,\vartheta) \leq 2 \eta_{k}\right\}\right]\\
&\nonumber \leq 
\sum_{k\in J}\eta_{k}^{-\beta p}~ \EEE\left[\sup~\left\{d_{\cX}(X_{\theta}, X_{\vartheta})^{p}~\Big|~ \theta,\vartheta\in\overline{\Theta},d_{\Theta}(\theta,\vartheta) \leq 2 \eta_{k}\right\}\right]\\
&\label{Ungleichung1} \leq 
\sum_{k\in J}2^{k \beta p}~ \EEE\left[\sup~\left\{d_{\cX}(X_{\theta},X_{\vartheta})^{p}~\Big|~ \theta,\vartheta\in\overline{\Theta},d_{\Theta}(\theta,\vartheta) \leq 2 \eta_{k}\right\}\right]
\end{align}
For $k\in J$ the application of Lemma \ref{Lemma B.2.7} yields
\begin{align}
&\nonumber\EEE\left[\sup~\left\{d_{\cX}(X_{\theta}, X_{\vartheta})^{p}~\Big|~ \theta,\vartheta\in\overline{\Theta},d_{\Theta}(\theta,\vartheta) \leq 2 \eta_{k}\right\}\right]\\
&\nonumber \leq 
4^{2 p + 4 q + 2}~M~\left(V_{k}(\Theta,\overline{\Theta})~ \left(2 \eta_{k}\right)^{q}~
+ 
\frac{C~\left(2 \eta_{k}\right)^{q-t}}{\big(2^{(q - t)/p} - 1\big)^{p}}\right)\\
&\label{Ungleichung2}\leq 
4^{2 p + 4 q + 2}~M~\big(\Delta(\Theta) + 1\big)^{q}~
\left(V_{k}(\Theta,\overline{\Theta})~ 2^{(-k + 1) q}~
+ 
\frac{C~2^{(-k + 1) (q-t)}}{\big(2^{(q - t)/p} - 1\big)^{p}}\right), 
 \end{align}
 where
 $$
V_{k}(\Theta,\overline{\Theta}) 
:= 
N\big(\overline{\Theta},d_{\Theta},\eta_{k+1}\big)~
\left[\ln\Big(N\big(\overline{\Theta},d_{\Theta},\eta_{k+1}\big)\Big)\right]^{q}.
$$
Moreover, the set $\{k\in\NNN\mid \eta_{k+1}\leq\Delta(\Theta)\}$ is nonvoid so that we may select its minimum say $k_{0}$. In view of (\ref{size of parameter space}) this means
\begin{eqnarray*}
N\big(\overline{\Theta},d_{\Theta},\eta_{k+1}\big) 
\leq  
C~ \left(\frac{2^{k + 1}}{\Delta(\Theta) + 1}\right)^{t}
\leq 
C~ 2^{(k + 1) t}\quad\mbox{for}~k\in\NNN, k\geq k_{0}.
\end{eqnarray*}
Hence for $k\in J$ with $k\geq k_{0}$ we may give a further upper estimate of inequality (\ref{Ungleichung2}) by
\begin{align*}
&\EEE\left[\sup~\left\{d_{\cX}(X_{\theta},X_{\vartheta})^{p}~\Big|~ \theta,\vartheta\in\overline{\Theta},d_{\Theta}(\theta,\vartheta) \leq 2 \eta_{k}\right\}\right]\\
&\leq 
4^{2 p + 4 q + 2}~2^{(-k + 1) (q-t)}~M~C~\big(\Delta(\Theta) + 1\big)^{q}~
\left(4^{t} \big[\ln\big(C\cdot 2^{(k + 1) t}\big)\big]^{q}
+ 
\frac{1}{\big(2^{(q - t)/p} - 1\big)^{p}}\right).
\end{align*}
Then 
\begin{align}
&\nonumber 2^{k \beta p}~\EEE\left[\sup~\left\{d_{\cX}(X_{\theta},X_{\vartheta})^{p}~\Big|~ \theta,\vartheta\in\overline{\Theta},d_{\Theta}(\theta,\vartheta) \leq 2 \eta_{k}\right\}\right]\\
&\label{Ungleichung3}\leq 
\dfrac{4^{2p + 5 q + 2}~M~C~\big(\Delta(\Theta) + 1\big)^{q}}{\big(2^{(q - t)/p} - 1\big)^{p}}~
2^{\left(\beta p - (q-t)\right) k}~\left(4^{t} \big[\ln\big(C\cdot 2^{(k + 1) t}\big)\big]^{q}\cdot \big(2^{(q - t)/p} - 1\big)^{p}
+ 
1\right)
\end{align}
holds for $k\in J$ with $k\geq k_{0}$. Next, setting 
$$
a_{k} := \dfrac{4^{2p + 5 q + 2}~M~C~\big(\Delta(\Theta) + 1\big)^{q}}{\big(2^{(q - t)/p} - 1\big)^{p}}~
2^{\left(\beta p - (q-t)\right) k}~\left(4^{t} \big[\ln\big(C\cdot 2^{(k + 1) t}\big)\big]^{q}\cdot \big(2^{(q - t)/p} - 1\big)^{p}
+ 
1\right)
$$
we may observe
$$
\lim_{k\to\infty}~\frac{|a_{k+1}|}{|a_{k}|} = 2^{\beta p - (q-t)}~\lim_{k\to\infty}
\dfrac{4^{t} \big[\ln\big(C\cdot 2^{(k + 2) t}\big)\big]^{q}\cdot \big(2^{(q - t)/p} - 1\big)^{p}
+ 1}{4^{t} \big[\ln\big(C\cdot 2^{(k + 1) t}\big)\big]^{q}\cdot \big(2^{(q - t)/p} - 1\big)^{p}
+ 1} =  2^{\beta p - (q-t)} < 1.
$$
Therefore 
\begin{align}
&
L_{1}(\Theta,C,t,M,p,q,\beta)
\notag\\
&:=
\frac{4^{2p + 5 q + 2}~M~C~\big(\Delta(\Theta) + 1\big)^{q}}{\big(2^{(q - t)/p} - 1\big)^{p}}~\sum_{k=k_{0}}^{\infty}
2^{\left(\beta p - (q-t)\right) k}~\left(4^{t} \big[\ln\big(C\cdot 2^{(k + 1) t}\big)\big]^{q}\cdot \big(2^{(q - t)/p} - 1\big)^{p}
+ 
1\right)
\notag\\
& < \infty.
\label{erster Koeffizient}
\end{align}
Moreover, by choice of $k_{0}$ we have 
$$
N\big(\overline{\Theta},d_{\Theta},\eta_{k+1}\big)  = 1~\mbox{if}~k\in\NNN, k < k_{0}
$$
which implies that $V_{k}(\Theta,\overline{\Theta}) = 0$ is valid for $k\in J$ with $k < k_{0}$. Then with $\sum\limits_{\emptyset} := 0$, the application of (\ref{Ungleichung2}) yields
\begin{align}
&\nonumber
\sum_{k=1\atop k\in J}^{k_{0} - 1}2^{k\beta p}~ \EEE\left[\sup~\left\{d_{\cX}(X_{\theta},X_{\vartheta})^{p}~\Big|~ \theta,\vartheta\in\overline{\Theta},d_{\Theta}(\theta,\vartheta) \leq 2^{-k +1}\big(\Delta(\Theta) + 1\big)\right\}\right]\\
&\nonumber \leq 
\frac{4^{2 p + 4 q + 2}~M~\big(\Delta(\Theta) + 1\big)^{q}}{\big(2^{(q - t)/p} - 1\big)^{p}}~
\sum_{k=1}^{k_{0} - 1}C~2^{(-k + 1) (q-t)}~2^{k\beta p}\\ 
&\label{Ungleichung4}\leq 
\frac{4^{2p + 5 q + 2 }~M~\big(\Delta(\Theta) + 1\big)^{q}}{\big(2^{(q - t)/p} - 1\big)^{p}}~
C~\sum_{k=1}^{\infty}2^{\left(\beta p - (q-t)\right)k}. 
\end{align}
Since $\beta p < q - t$ we obtain that
\begin{equation}
\label{zweiter Koeffizient}
L_{2}(\Theta,C,t,M,p,q,\beta)
: = \frac{4^{2p + 5 q + 2 }~M~\big(\Delta(\Theta) + 1\big)^{q}}{\big(2^{(q - t)/p} - 1\big)^{p}}~
C~\sum_{k=1}^{\infty}2^{\left(\beta p - (q-t)\right)k} < \infty.
\end{equation}
Combining (\ref{Ungleichung1}), (\ref{Ungleichung3}) and (\ref{Ungleichung4}) with (\ref{erster Koeffizient}) and (\ref{zweiter Koeffizient}), we end up with 
$$
\EEE\left[\sup_{\theta,\vartheta\in\overline{\Theta}\atop \theta\not=\vartheta}~\frac{d_{\cX}(X_{\theta}, X_{\vartheta})^{p}}{d_{\Theta}(\theta,\vartheta)^{\beta p}}\right]\leq \sum_{j=1}^{2}
L_{j}(\Theta,C,t,M,p,q,\beta)
=:L(\Theta,C,t,M,p,q,\beta).
$$
This yields the first claim of Proposition~\ref{Kolmogorov-Chentsov_finite}.
The second claim is a direct consequence of the expressions in (\ref{erster Koeffizient}) and~(\ref{zweiter Koeffizient}).
\end{proof}

\noindent
\textbf{Proof of Theorem~\ref{Kolmogorov-Chentsov}}
We first fix any $\beta\in ]0,(q-t)/p[$.
Let the constant
$L(\Theta,C,t,M,p,q,\beta)$
be chosen according to Proposition~\ref{Kolmogorov-Chentsov_finite}, and let us consider any at most countable subset $\overline{\Theta}$ of $\Theta$ which consists of at least two elements $\overline{\theta}, \overline{\vartheta}$. We may select some sequence $(\overline{\Theta}_{k})_{k\in\Theta}$ of nonvoid finite subsets of $\overline{\Theta}$ with at least two elements satisfying 
$$
\overline{\theta}, \overline{\vartheta}\in\overline{\Theta}_{k}\subseteq\overline{\Theta}_{k + 1}\quad\mbox{for}~k\in\NNN\quad\mbox{and}\quad\bigcup_{k= 1}^{\infty}\overline{\Theta}_{k} = \overline{\Theta}.
$$
Then 
$$
\sup_{\theta,\vartheta\in\overline{\Theta}_{k}\atop \theta\not=\vartheta}
\frac{d_{\cX}(X_{\theta},X_{\vartheta})^{p}}{d_{\Theta}(\theta,\vartheta)^{\beta p}}
\nearrow \sup_{k\in\NNN}\sup_{\theta,\vartheta\in\overline{\Theta}_{k}\atop \theta\not=\vartheta }\frac{d_{\cX}(X_{\theta},X_{\vartheta})^{p}}{d_{\Theta}(\theta,\vartheta)^{\beta p}} = \sup_{\theta,\vartheta\in\overline{\Theta}\atop \theta\not=\vartheta}\frac{d_{\cX}(X_{\theta},X_{\vartheta})^{p}}{d_{\Theta}(\theta,\vartheta)^{\beta p}},
$$
and thus by monotone convergence theorem along with Proposition~\ref{Kolmogorov-Chentsov_finite}
\begin{equation}
\label{monotoneKonvergenzUngleichung}
\EEE\left[\sup_{\theta,\vartheta\in\overline{\Theta}\atop \theta\not= \vartheta}\frac{d_{\cX}(X_{\theta}, X_{\vartheta})^{p}}{d_{\Theta}(\theta,\vartheta)^{\beta p}}\right]
=
\lim_{k\to\infty}\EEE\left[\sup_{\theta,\vartheta\in\overline{\Theta}_{k}\atop \theta \not = \vartheta}\frac{d_{\cX}(X_{\theta}, X_{\vartheta})^{p}}{d_{\Theta}(\theta,\vartheta)^{\beta p}}\right]
\le
L(\Theta,C,t,M,p,q,\beta).
\end{equation}
This shows \eqref{Hoelder-constant2} due to the second statement of Proposition \ref{Kolmogorov-Chentsov_finite}.

For the remaining part of the proof let us assume that $d_{\cX}$ is complete, and let $\overline{\Theta}$ be some at most countable subset of $\Theta$ which is dense w.r.t.\ $d_{\Theta}$. As a further consequence of (\ref{monotoneKonvergenzUngleichung}) we have $\MP(A) = 1$, where
$$
A := \left\{\sup_{\theta,\vartheta\in\overline{\Theta}\atop \theta\not= \vartheta}\frac{d_{\cX}(X_{\theta}, X_{\vartheta})^{p}}{d_{\Theta}(\theta,\vartheta)^{\beta p}} < \infty\right\}.
$$
This implies that on $A$ the  random process $(X_{\theta})_{\theta\in\overline{\Theta}}$ has 
H\"older-continuous paths of order $\beta$. By completeness of $d_{\cX}$ we may define a new random process $(\overline{X}_{\theta})_{\theta\in\Theta}$ via
$$
\overline{X}_{\theta}(\omega) := 
\begin{cases}
\lim\limits_{\vartheta\to\theta\atop \vartheta\in\overline{\Theta}}X_{\vartheta}(\omega),&\omega\in A,\\
\bar x,&\omega\not\in A,
\end{cases}
$$
where $\bar x\in\cX$ is arbitrary.
Clearly, this process has H\"older-continuous paths of order~$\beta$.
Furthermore, it can be shown by standard arguments that this random process satisfies~(\ref{eq:14122019a1}).
We now show that it is a modification of $(X_{\theta})_{\theta\in\Theta}$. For this purpose let us fix any $\theta\in\Theta$, and let $(\vartheta_{k})_{k\in\NNN}$ be a sequence from $\overline{\Theta}$ which converges to $\theta$ w.r.t.\ $d_{\Theta}$. By construction of $(\overline{X}_{\theta})_{\theta\in\Theta}$ we may invoke inequality (\ref{increment inequality}) to conclude
\begin{eqnarray*}
\EEE\left[d_{\cX}(X_{\theta},\overline{X}_{\vartheta_{k}})^{p}\right]
\leq 
\EEE\left[\eins_{A}\cdot d_{\cX}(X_{\theta}, X_{\vartheta_{k}})^{p}\right]
\leq 
M~d_{\Theta}(\theta,\vartheta_{k})^{q}\to 0\quad\mbox{for}~k\to\infty.
\end{eqnarray*}
In particular, on the one hand the sequence $\big(d_{\cX}(X_{\theta},\overline{X}_{\vartheta_{k}})\big)_{k\in\NNN}$ converges in probability to $0$. On the other hand by definition of $(\overline{X}_{\theta})_{\theta\in\Theta}$, the sequence $\big(d_{\cX}(\overline{X}_{\theta}, \overline{X}_{\vartheta_{k}})\big)_{k\in\NNN}$ converges in probability to $0$. Then if $l\in\NNN$
\begin{eqnarray*}
0
&\leq& 
\limsup_{k\to\infty}\MP\big(\big\{d_{\cX}(X_{\theta}, \overline{X}_{\theta}) > l\big\}\big)\\
&\leq&
\limsup_{k\to\infty}\MP\big(\big\{d_{\cX}(X_{\theta}, X_{\vartheta_{k}}) + d_{\cX}(X_{\vartheta_{k}}, \overline{X}_{\theta}) > l\big\}\big)\\
&\leq& 
\limsup_{k\to\infty}\Big[\MP\big(\big\{d_{\cX}(X_{\theta}, X_{\vartheta_{k}}) > l/2\big\}\big) + 
\MP\big(\big\{d_{\cX}(\overline{X}_{\vartheta_{k}},\overline{X}_{\theta}) > l/2\big\}\big)\Big] = 0,
\end{eqnarray*}
and thus
$$
\MP\big(\big\{d_{\cX}(X_{\theta},\overline{X}_{\theta}) > 0\big\}\big) 
= 
\lim_{l\to\infty}\MP\big(\big\{d_{\cX}(X_{\theta}, \overline{X}_{\theta}) > l\big\}\big) = 0.
$$
Hence $\MP\big(\big\{X_{\theta} \not= \overline{X}_{\theta}\big\}\big) = 0$,
i.e., $(\overline X_\theta)_{\theta\in\Theta}$ is a modification of $(X_\theta)_{\theta\in\Theta}$.

\medskip
Finally, consider an increasing sequence $(\beta_n)_{n\in\NNN}\subset\:]0,(q-t)/p[$ such that $\beta_n\to(q-t)/p$, as $n\to\infty$.
The argument above shows that, for any $n\in\NNN$, the process $(X_\theta)_{\theta\in\Theta}$ has a modification
$(\overline X^n_\theta)_{\theta\in\Theta}$
with H\"older-continuous paths of order $\beta_n$ and satisfying~(\ref{eq:14122019a1}).
Let us fix for a moment an arbitrary $n\in\NNN$.
The processes $(\overline X^n_\theta)_{\theta\in\Theta}$ and $(\overline X^{n+1}_\theta)_{\theta\in\Theta}$ are indistinguishable because they are modifications of each other, both continuous, and $\Theta$ is separable (as a totally bounded metric space).
We can, therefore, find an event $\Omega_n\in\cF$ with $\pr(\Omega_n)=1$ such that, for all $\omega\in\Omega_n$ and $\theta\in\Theta$, it holds
$\overline X^n_\theta(\omega)=\overline X^{n+1}_\theta(\omega)$.
We then define the set
$$
\Omega_\infty=\bigcap_{n\in\NNN}\Omega_n
$$
and notice that $\pr(\Omega_\infty)=1$ and, for all $\omega\in\Omega_\infty$, $\theta\in\Theta$ and $n\in\NNN\setminus\{1\}$, it holds
$$
\overline X^1_\theta(\omega)=\overline X^n_\theta(\omega).
$$
Consequently, the process $(\widetilde X_\theta)_{\theta\in\Theta}$ defined via
$$
\widetilde X_\theta(\omega)=\begin{cases}
\overline X^1_\theta(\omega),&\omega\in\Omega_\infty,\\
\bar x,&\omega\notin\Omega_\infty,
\end{cases}
$$
where $\bar x\in\cX$ is arbitrary, is a modification of $(X_\theta)_{\theta\in\Theta}$ such that \emph{all} its paths are H\"older-continuous of \emph{all} orders $\beta\in]0,(q-t)/p[$. Note that $(\widetilde X_\theta)_{\theta\in\Theta}$ also satisfies~(\ref{eq:14122019a1}). This concludes the proof.
\hfill$\Box$

\appendix
\section{Criterion for uniform tightness}
\label{appendix tightness}
Let $(\Theta,d_{\Theta})$ be a compact metric space and $(\cX,d_{\cX})$ a complete metric space.
We use the notation $w(\cdot,\cdot)$ for the modulus of continuity on $\cC(\Theta,\cX)$, i.e.,
$$
w\colon\cC(\Theta,\cX)\times ]0,\infty[\rightarrow\RRR,\quad(f,\delta)\mapsto \sup_{d_{\Theta}(\theta,\vartheta) \leq \delta\atop \theta,\vartheta\in\Theta}d_{\cX}\big(f(\theta), f(\vartheta)\big).
$$
To make the paper self-contained, we present here the precise formulation of the criterion for uniform tightness,
which was applied in the proof of Proposition~\ref{vereinfachtes Kriterium fuer Straffheit}.

\begin{theorem}
\label{schwache Konvergenz im Pfadraum}
Let $(X_{n})_{n\in\NNN}$ be a sequence of Borel random elements $X_{n}:\Omega\to\cC(\Theta,\cX)$ on some probability space $\OFP$.
Let $\Theta'\subseteq\Theta$ be dense in $\Theta$.
The sequence $(X_{n})_{n\in\NNN}$ is uniformly tight if and only if
\begin{equation}
\label{fd-convergence}
\big(X_{n}(\cdot,\theta)\big)_{n\in\NNN}\text{ is a uniformly tight sequence of random elements in }\cX,\text{ for all }\theta\in\Theta',
\end{equation}
and, for every $\eps>0$,
\begin{equation}
\label{Straffheitsbedingung}
\lim_{\delta\to 0+}\limsup_{n\to\infty}\,\pr\left(w(X_n,\delta)\ge\eps\right) = 0.
\end{equation}
\end{theorem}

We remark that since, for any $\delta>0$, the mapping $w(\cdot,\delta)$ is continuous, $w(X_{n},\delta)$ is a random variable for every $n\in\NNN$
(in particular, the probability in~(\ref{Straffheitsbedingung}) is well-defined).

\medskip\noindent
\begin{proof}
The result is well-known for $\cX = \RRR$ (see \cite[Theorem~1.5.7]{vanderVaartWellner1996}), and a similiar one is shown in the proof of Theorem~14.5 from \cite{Kallenberg1997} in the case of separable and complete $d_{\cX}$. We shall use arguments from the proof of Theorem~14.5 in \cite{Kallenberg1997}.

\medskip
Firstly, let us assume that $(X_{n})_{n\in\NNN}$ is a uniformly tight sequence in $\cC(\Theta,\cX)$, and let $\varepsilon, \eta\in ]0,1[$. Then there exists some compact subset ${\cal K}\subseteq \cC(\Theta,\cX)$ such that 
\begin{equation}
\label{Eins}
\sup_{n\in\NNN}\MP(\{X_{n}\in\cC(\Theta,\cX)\setminus {\cal K}\}) \le \eta.
\end{equation}
By a general version of the Arzela-Ascoli theorem (see \cite[Theorem~47.1]{Munkres2000}) the set ${\cal K}$ is equicontinuous which means that we may find for any $\theta\in\Theta$ some $\delta_{\theta} > 0$ such that $d_{\cX}\big(f(\vartheta),f(\theta)\big) < \varepsilon/2$ for every $f\in {\cal K}$ whenever $d_{\theta}(\vartheta,\theta) < \delta_{\theta}$. Since $\Theta$ is compact we may apply Lebesgue's number lemma (see \cite[Lemma~27.5]{Munkres2000}) to the open cover of $\Theta$ consisting of the open metric balls $B_{\delta_{\theta}}(\theta)$ around $\theta$ with radius $\delta_{\theta}$. In this way we may select some $\delta > 0$ such that $w(f,\delta) < \varepsilon$ holds for every $f\in {\cal K}$. Then, for all $n\in\NNN$,
\begin{eqnarray*}
\MP(\{w(X_{n},\delta)\geq\varepsilon\}) 
\le
\MP(\{X_{n}\in\cC(\Theta,\cX)\setminus {\cal K}\}) \le \eta,
\end{eqnarray*}
and thus
\begin{equation}
\label{Zwei}
\lim_{\delta\to 0+}\limsup_{n\to\infty}\MP\big(\big\{w(X_{n},\delta)\geq\varepsilon\big\}\big)\le\eta.
\end{equation}
Moreover, for $\theta\in\Theta$ the projection mapping 
$$
\pi_\theta\colon\cC(\Theta,\cX)\to\cX,\quad f\mapsto f(\theta)
$$
is continuous so that $\pi_{\theta}(\mathcal{K})$ is a compact subset of $\cX$. Furthermore by (\ref{Eins})
\begin{equation}
\label{Drei}
\sup_{n\in\NNN}\MP\big(\{X_{n}(\cdot,\theta)\in\cX\setminus \pi_{\theta}({\cal K})\}\big))\leq\sup_{n\in\NNN}\MP(\{X_{n}\in\cC(\Theta,\cX)\setminus {\cal K}\}) \le \eta.
\end{equation}
Then we conclude from (\ref{Zwei}) and (\ref{Drei}) that $(X_{n})_{n\in\NNN}$ satisfies conditions (\ref{fd-convergence}) and~(\ref{Straffheitsbedingung}) because $\varepsilon,\eta$ were chosen arbitrarily in $]0,1[$. This finishes the proof of the only if part.

\medskip
For the if part, let $(X_{n})_{n\in\NNN}$ fulfill conditions (\ref{fd-convergence}) and~(\ref{Straffheitsbedingung}). Fix any $\gamma \in ]0,1[$. Since $\Theta$ is compact, the mappings $X_{n}(\omega,\cdot)$ are uniformly continuous for $\omega\in\Omega$ and $n\in\NNN$. Hence, for $\omega\in\Omega$ and $n\in\NNN$, we have $w(X_{n},\delta)\to 0$, as $\delta\to 0$.
Combining this observation with condition~(\ref{Straffheitsbedingung}), we may find for any $k\in\NNN$ some $\overline{\delta}_{k} >  0$ such that 
\begin{equation}
\label{Stetigkeitsmodulus Abschaetzung}
\sup_{n\in\NNN}\MP\big(\big\{w(X_{n},\overline{\delta}_{k}) > 2^{-k}\big\}\big)\leq 2^{- k -1 }~\gamma.
\end{equation}
Since $\Theta$ is compact, the metric on $\Theta'$ is separable. In addition $\Theta'$ is dense. Then, there is some sequence $(\theta_{k})_{k\in\NNN}$ in $\Theta'$ which is a dense subset of $\Theta$. Hence in view of~(\ref{fd-convergence}) we may find for every $k\in\NNN$ some compact subset $K_{k}$ of $\cX$ such that
\begin{equation}
\label{Straffheit endlichdimensional}
\sup_{n\in\NNN}\MP\big(\big\{X_{n}(\cdot,\theta_{k})\not\in K_{k}\}\big)\leq 2^{- k -1 }~\gamma.
\end{equation}
Now define
$$
B := \bigcap_{k=1}^{\infty}\left\{f\in \cC(\Theta,\cX)\mid f(\theta_{k})\in K_{k}, w(f,\overline{\delta}_{k})\leq 2^{-k}\right\},
$$
and denote the topological closure of $B$ w.r.t.\ $d_{\infty}$ by $cl(B)$. Then 
we have
\begin{align*}
&
\sup_{n\to\infty}\MP\left(\left\{X_{n}\in \cC(\Theta,\cX)\setminus cl(B)\right\}\right)\\
&\leq 
\sup_{n\to\infty}\MP\left(\left\{X_{n}\in \cC(\Theta,\cX)\setminus B\right\}\right)\\
&\leq
\sup_{n\to\infty} 
\sum_{k=1}^{\infty}\left[\MP(\{X_{n}(\cdot,\theta_{k})\in\cX\setminus K_{k}\}) + \MP(\{w(X_{n},\overline{\delta}_{k}) > 2^{-k}\})\right].
\end{align*}
So invoking~(\ref{Stetigkeitsmodulus Abschaetzung}) along with~(\ref{Straffheit endlichdimensional}), we end up with 
$$
\sup_{n\to\infty}\MP\left(\left\{X_{n}\in \cC(\Theta,\cX)\setminus cl(B)\right\}\right) 
\leq 
 \sum_{k=1}^{\infty}\left[2^{- k - 1}~\gamma + 2^{-k - 1}~\gamma\right] = \gamma.
$$
Hence it is left to show that $B$ is a relatively compact subset of $\cC(\Theta,\cX)$.

\medskip
For an arbitrary $\varepsilon > 0$ and for every $\theta\in\Theta$,
choose some $k\in\NNN$ such that $2^{-k} < \varepsilon$. Then by construction $B$, we obtain $w(f,\overline{\delta}_{k}) < \varepsilon$ for $f\in B$. In particular, $d_{\cX}\big(f(\vartheta),f(\theta)\big) < \varepsilon$ for every $f\in B$ and any $\vartheta\in\Theta$ with $d_{\Theta}(\vartheta,\theta)< \overline{\delta}_{k}$. Thus we have shown that $B$ is some equicontinuous subset of $\cC(\Theta,\cX)$. Therefore by a general version of the Arzela-Ascoli theorem \cite[Theorem~47.1]{Munkres2000} it remains to show that the set 
$\{f(\theta)\mid f\in B\}$ is a relatively compact subset of $\cX$ for any $\theta\in\Theta$. 
This means to show that this set is totally bounded w.r.t.\ $d_{\cX}$ due to completeness of $d_{\cX}$.

\medskip
Let us fix any $\theta\in\Theta$. Choose, for an arbitrary $\varepsilon > 0$, some $k_{0}\in\NNN$ such that $2^{-k_{0}} < \varepsilon/2$. Since $\{\theta_{k}\mid k\in\NNN\}$ is dense, we may find some $k_{1}\in\NNN$ such that $d_{\Theta}(\theta_{k_{1}},\theta) < \overline{\delta}_{k_{0}}$. This implies by the construction of the set $B$ that
\begin{equation}
\label{flach}
f(\theta_{k_{1}})\in K_{k_{1}}\quad\mbox{and}\quad d_{\cX}\big(f(\theta_{k_{1}}),f(\theta)\big) < \varepsilon/2\quad\mbox{for every}~ f\in B.
\end{equation} 
The set $K_{k_{1}}$ is assumed to be compact, in particular, it is totally bounded w.r.t.\ $d_{\cX}$, so that there exist $m\in\NNN$ and $x_{1},\dots,x_{m}\in\cX$ satisfying
$$
K_{k_{1}}\subseteq \bigcup_{i=1}^{m}B_{\varepsilon/2}(x_{i}),
$$
where, for $r>0$, $B_r(x_{i})$ denotes the open $d_{\cX}$-metric ball around $x_{i}$ with radius $r$. Then we may conclude from~(\ref{flach})
$$
\{f(\theta)\mid f\in B\}\subseteq \bigcup_{i=1}^{m}B_{\varepsilon}(x_{i}).
$$
This shows that the set $\{f(\theta)\mid f\in B\}$ is totally bounded w.r.t.\ $d_{\cX}$, which completes the proof.
\end{proof}

\textbf{Acknowledgement:}
We thank the anonymous referee for many constructive comments and suggestions that helped improve the paper.

\section*{Declarations}

\textbf{Conflicts of interest:} The authors have no competing interests to declare that are relevant to the content of this article.

\medskip\noindent
\textbf{Funding:} The authors did not receive support from any organization for the submitted work.


\begin{thebibliography}{99}
\bibitem{AliprantisBorder2006} C. D. Aliprantis and K. C. Border, {\em Infinite dimensional analysis}, Springer, Berlin, 2006 (Third edition).
\bibitem{AKU} Ankirchner, S., Kruse, T. and Urusov, M. (2021).
Wasserstein convergence rates for random bit approximations of continuous Markov processes.
{\em Journal of Mathematical Analysis and Applications.}
493 (2), 124543.
\bibitem{AraujoGine1980} Araujo, A. and Gine, E. (1980). {\em The central limit theorem for real and Banach valued random variables.} Wiley, New York.
\bibitem{Chentsov1956} Chentsov, N. N. (1957). Weak convergence of stochastic processes whose trajectories have no discontinuities of the second kind and the ``heuristic'' approach to the Kolmogorov-Smirnov tests. {\em Theor. Prob. Appl.} 1, 140--144.
\bibitem{ChobanjanTarieladze1977} Chobanjan, S. A. and Tarieladze, V. I. (1977). Gaussian characterizations of certain Banach spaces. {\em Journal of Multivariate Analysis} 7, 183--203. 
\bibitem{CoxEtAl2013} Cox, S. G., Hutzenthaler, M. and Jentzen, A. (2021). Local Lipschitz continuity in the initial value and strong completeness for nonlinear stochastic differential equations.
{\em Preprint, arXiv:1309.5595v3}, 1--96,
Accepted in {\em Memoires of the American Mathematical Society}.
\bibitem{Dieudonne1972} Dieudonn\'{e}, J. (1972). {\em Treatise on analysis. Volume III}. Academic Press, New York, San Francisco, London.
\bibitem{doCarmo1992} do Carmo, M. P. (1992). {\em Riemannian geometry}. Birkh\"auser, Boston, Basel, Berlin.
\bibitem{GikhmanSkorokhod1974}Gikhman, I. I. and Skorokhod, A. V. (1974). 
{\em Theory of random processes. Vol.~1}. Springer, Berlin, Heidelberg, New York.
\bibitem{GrohsEtAl2018} Grohs, P., Hornung, F., Jentzen, A. and von Wurstemberger. P. (2018). A proof that artificial neural networks overcome the curse of dimensionality in the numerical approximation of Black-Scholes partial differential equations. {\em Preprint, arXiv:1809.02362v1}, 1--124.
\bibitem{Hoffmann-JorgensenPisier1976} Hoffmann-Jorgensen, J. and Pisier, G. (1976).  The law of large numbers and the central limit theorem in Banach spaces. {\em Annals of Probability} 4, 587--599.
\bibitem{HuddeEtAl2019} Hudde, A., Hutzenthaler, M. and Mazzonetto, S. (2019). A stochastic Gronwall inequality and applications to moments, strong completeness, strong local Lipschitz continuity, and perturbations. {\em Preprint, arXiv:1903.08727v2}, 1--26.
\bibitem{Jain1977} Jain, N. C. (1977). Central limit theorem and related questions in Banach spaces, in: J. L. Doob (ed.), Proc. Symp. in Pure Mathematics XXXI, 55--66, Am. Math. Soc., Providence.
\bibitem{Jost2011} Jost, J. (2011). {\em Riemannian geometry and geometric analysis}. Springer, Berlin, Heidelberg, New York.
\bibitem{Kallenberg1997} Kallenberg, O. (1997). \textit{Foundation of modern probability}. Springer, New York.
\bibitem{KobayashiNomizu1963} Kobayashi, S. and Nomizu, K. (1963). {\em Foundations of differential geometry I}. Interscience Publishers,
John Wiley \& Sons, New York, London.
\bibitem{LangEtAl2016} Lang, A., Potthoff, J., Schlather, M. and Schwab, D. (2016). Continuity of random fields on Riemannian manifolds. {\em Commun. Stoch. Anal.} 10 (2), 185--193.
\bibitem{LedouxTalagrand1991} Ledoux, M. and Talagrand, M. (1991). {\em Probability in Banach spaces.} Springer, Berlin.
\bibitem{MittmannSteinwart2003} Mittmann, K. and Steinwart, I. (2003). On the existence of continuous modifications of vector-valued random fields. {\em Georgian Mathematical Journal} 10, 311--317.
\bibitem{Munkres2000} Munkres, J. R. (2000). {\em Topology}, Prentice Hall, Englewood Cliffs, N.J. (2nd ed.).
\bibitem{Potthoff2009} Potthoff, J. (2009). Sample properties of random fields. II. Continuity. {\em Commun. Stoch. Anal.} 3 (3), 331--348.
\bibitem{RevuzYor1999} Revuz, D. and Yor, M. (1999). {\em Continuous martingales and Brownian motion}. Springer, Berlin. (3rd ed.).
\bibitem{Talagrand2014} Talagrand, M. (2014). \textit{ Upper and lower bounds for stochastic processes}. Springer, Heidelberg.
\bibitem{VakhaniaEtAl.1987} Vakhania, N. N., Tarieladze, V. I. and Chobanyan, S. A. (1987). \textit{Probability distributions on Banach spaces}, Kluwer, Dordrecht.
\bibitem{vandeGeer2000} van de Geer, S. (2000). \textit{Empirical processes in M-estimation}. Cambridge University Press, Cambridge.
\bibitem{vanderVaartWellner1996} van der Vaart, A.W. and Wellner, J.A. (1996). \textit{Weak convergence and empirical processes}. Springer, New York.
\end{thebibliography}
\end{document}